\documentclass[hidelinks, 11pt]{amsart}

\usepackage[margin=3cm]{geometry}

\usepackage{amssymb}

\usepackage{hyperref}
\usepackage[all]{xy}
\usepackage{enumerate}
\usepackage{enumitem}
\usepackage{bbm}
\usepackage{tikz-cd}
\usepackage{extpfeil}
\usepackage{relsize}
\usepackage[bbgreekl]{mathbbol}
\usepackage{slashed}
\usepackage{amsfonts}
\usepackage{bm}
\usepackage{etoolbox}
\usepackage{soul}
\usepackage{mathrsfs}
\usepackage{adjustbox}

\usetikzlibrary{decorations.pathmorphing,arrows.meta}

\usepackage{xcolor}
\colorlet{BLUE}{blue}

\DeclareMathOperator{\Set}{{Set}}
\DeclareMathOperator{\Pro}{{Pro}}

\newcommand{\ZZ}{{\mathbf{Z}}}
\newcommand{\QQ}{{\mathbf{Q}}}

\newcommand{\BB}{{\mathrm{B}}}

\newcommand{\bB}{{\mathbb B}}
\newcommand{\bC}{{\mathbb C}}

\newcommand{\bF}{{\mathbb F}}
\newcommand{\bG}{{\mathbb G}}

\newcommand{\bP}{{\mathbb P}}
\newcommand{\bQ}{{\mathbb Q}}
\newcommand{\bR}{{\mathbb R}}

\newcommand{\bZ}{{\mathbb Z}}

\newcommand{\cA}{{\mathcal A}}

\newcommand{\cC}{{\mathcal C}}
\newcommand{\cD}{{\mathcal D}}

\newcommand{\cL}{{\mathcal L}}

\newcommand{\cO}{{\mathcal O}}

\newcommand{\cX}{{\mathcal X}}

\newcommand{\fX}{{\mathfrak X}}

\newcommand{\fm}{{\mathfrak m}}

\newcommand{\oK}{\overline{K}}

\newcommand{\obQ}{\overline{\mathbb{Q}}}

\newcommand{\tx}{\widetilde{x}}
\newcommand{\ty}{\widetilde{y}}

\newcommand{\hH}{\mathrm{H}}
\newcommand{\rR}{\mathrm{R}}
\newcommand{\rE}{\mathrm{E}}
\newcommand{\nr}{\mathrm{nr}}
\newcommand{\SL}{\mathrm{SL}}

\newcommand{\DD}{\mathbf{D}}

\DeclareMathOperator*{\colim}{colim}
\newcommand{\qproet}{\mathrm{qpro\acute{e}t}}
\newcommand{\fet}{\mathrm{f\acute{e}t}}

\DeclareMathOperator{\ab}{{ab}}

\DeclareMathOperator{\gr}{{gr}}

\DeclareMathOperator{\Sym}{{Sym}}

\DeclareMathOperator{\Mod}{{Mod}}

\newcommand{\RGamma}{\mathrm{R}\Gamma}

\DeclareMathOperator{\Aut}{{Aut}}

\DeclareMathOperator{\Ext}{{Ext}}
\DeclareMathOperator{\Hom}{{Hom}}

\DeclareMathOperator{\id}{{id}}

\DeclareMathOperator{\op}{{op}}

\DeclareMathOperator{\Vect}{{Vect}}

\DeclareMathOperator{\coker}{coker}

\DeclareMathOperator{\Map}{Map}
\DeclareMathOperator{\Rep}{Rep}

\DeclareMathOperator{\Shv}{{Shv}}

\newcommand{\ud}{\underline}

\newcommand{\qc}{\mathrm{qc}}
\newcommand{\sep}{\mathrm{sep}}

\newcommand{\mix}{\mathrm{mix}}
\newcommand{\pure}{\mathrm{pure}}

\DeclareMathOperator{\WD}{WD}

\newcommand{\ip}{\frac{1}{p}}

\newcommand{\Fun}{\operatorname{Fun}}

\newcommand{\dR}{\mathrm{dR}}

\newcommand{\cris}{\mathrm{cris}}

\DeclareMathOperator{\Spa}{{Spa}}

\DeclareMathOperator{\Tot}{Tot}

\newcommand{\et}{\mathrm{\acute{e}t}}

\newcommand{\fil}{\mathrm{fil}}

\newcommand{\Fil}{\mathrm{Fil}}

\newcommand{\oF}{\overline{F}}

\newcommand{\an}{\mathrm{an}}
\newcommand{\proet}{\mathrm{pro\acute{e}t}}

\newcommand{\RHom}{\mathrm{RHom}}

\newcommand{\cont}{\mathrm{cont}}
\newcommand{\sing}{\mathrm{sing}}

\newcommand{\pst}{\mathrm{pst}}

\newcommand{\stab}{{\mathrm{st}}}
\newcommand{\FF}{{\rm FF}}
\newcommand{\Coh}{{\rm Coh}}
\newcommand{\obl}{{\rm obl}}
\newcommand{\MHS}{{\rm MHS}}

\DeclareSymbolFontAlphabet{\mathbb}{AMSb} 
\DeclareSymbolFontAlphabet{\mathbbl}{bbold}

\newcommand{\comment}[1]{}

\makeatletter

\newcommand{\Rmnum}[1]{\expandafter\@slowromancap\romannumeral #1@}
\makeatother

\numberwithin{equation}{section}

\newtheorem{pr}[equation]{Proposition}
\newtheorem{thm}[equation]{Theorem}

\newtheorem{lm}[equation]{Lemma}
\newtheorem{lemma}[equation]{Lemma}

\newtheorem{cor}[equation]{Corollary}

\theoremstyle{definition}

\newtheorem{rem}[equation]{Remark}

\newtheorem{construction}[equation]{Construction}
\newtheorem{notation}[equation]{Notation}
\newtheorem{question}[equation]{Question}

\newtheorem{examples}[equation]{Examples}

\newtheorem{defn}[equation]{Definition}

\newcommand{\com}[1]{\ifstrequal{0}{1}{\textcolor{blue}{#1}}{}}

\title[Formality for rigid-analytic spaces]{Formality for rigid-analytic spaces satisfying \newline the weight-monodromy conjecture}

\author{Alexander Petrov}
\address{Yale University, Department of Mathematics, 219 Prospect St, New Haven, CT 06511}
\email{alexander.petrov.57@gmail.com}

\author{Bogdan Zavyalov}
\address{University of Maryland, Department of Mathematics, 4176 Campus Dr, College Park, MD 20742}
\email{bogd.zavyalov@gmail.com}

\begin{document}
\renewcommand{\mathbb}{\mathbf}

\begin{abstract} We prove that \'etale and de Rham cohomology algebras of a smooth proper rigid-analytic space over a finite extension of $\bQ_p$ are formal if the rigid-analytic space satisfies the weight-monodromy conjecture. This is achieved by showing that the underlying $E_{\infty}$-algebra of a monodromy-pure $E_{\infty}$-algebra in Weil--Deligne representations is formal. We give examples of smooth proper rigid-analytic surfaces whose cohomology algebras are not formal.
\end{abstract}

\maketitle

\tableofcontents

\section{Introduction}

In this paper, we study formality of the cohomology of $p$-adic rigid-analytic spaces. Recall that complex Hodge theory implies the following two special properties of the cohomology of compact K\"ahler manifolds:

\begin{thm}[{\hspace{1sp}\cite{deligne-griffiths-morgan-sullivan}}]\label{thm:hodge-theory-complex} Let $X$ be a compact K\"ahler manifold.
\begin{enumerate}
    \item\label{thm:hodge-theory-complex-1} Then the Hodge-to-de Rham spectral sequence $\rE^{i, j}_1 = \hH^j(X, \Omega^i_X) \Rightarrow \hH^{i+j}(X, \bC)$ degenerates on the first page, and 
    \item\label{thm:hodge-theory-complex-2} the $E_\infty$-algebra $\rR\Gamma_{\dR}(X/\bC)$ is formal\footnote{We note that the notion of $E_\infty$-algebras over a field of characteristic $0$ is equivalent to the notion of commutative dg-algebras.}, i.e. it is quasi-isomorphic to the $E_\infty$-algebra $\bigoplus_{i\geq 0} \hH^i_{\dR}(X/\bC)[-i]$.
\end{enumerate}
\end{thm}

We note that \cite[Theorem 12.1]{Sullivan-infinitesimal}  and \cite[Corollary 6.9]{Obstructions} imply that the $E_\infty$-algebra $\rR\Gamma_\sing(X, \bQ)$ is also formal for a compact K\"ahler manifold $X$. The significance of this property is that for a simply connected K\"ahler manifold $X$, its rational homotopy type is therefore recovered from the cohomology ring $\hH^*_{\sing}(X,\bQ)$ via Quillen's rational homotopy theory.

For a prime number $p$, Hodge-to-de Rham degeneration holds for {\it all} smooth proper rigid-analytic spaces over $\bC_p\coloneqq \widehat{\obQ}_p$, cf. \cite[Theorem 2.1]{scholze-rio}. This is remarkable in that the degeneration fails for arbitrary compact complex manifolds, so in that respect arbitrary smooth proper rigid-analytic spaces behave as K\"ahler manifolds.

The main result of this paper is that the formality of the cohomology of a smooth proper rigid-analytic space defined over a finite extension of $\bQ_p$ holds, provided that it satisfies the weight-monodromy conjecture:
\begin{thm}[{{Theorem~\ref{thm:main-formality} and Lemma~\ref{dr vs padic etale}}}]\label{thm: wm implies formality intro}{Let $p, \ell$ be prime numbers, possibly equal, let $K$ be a finite extension of $\QQ_p$, and let $X$ be a smooth proper rigid-analytic space over $K$. Assume that the weight-monodromy conjecture holds for $\hH^i_\et(X_{\bC_p}, \QQ_\ell)$ for all $i\geq 0$. Then the $E_{\infty}$-algebra $\RGamma_{\et}(X_{\bC_p},\bQ_{\ell})$ is formal. For $\ell=p$, the de Rham cohomology $E_{\infty}$-algebra $\RGamma_{\dR}(X/K)$ is also formal.}
\end{thm}

{We refer to Definition~\ref{defn:weight-monodromy} for the precise meaning of the phrase ``the weight-monodromy conjecture holds for $\hH^i_\et(X_{\bC_p}, \QQ_\ell)$''. We also note that the example of a Hopf surface \cite[p. 16]{scholze-cdm} shows that the weight-monodromy conjecture is false for a general smooth proper rigid-analytic space. }Moreover, we find examples of smooth proper rigid-analytic spaces whose cohomology algebras are not formal:

\begin{thm}[{Proposition~\ref{main}}]\label{thm:main-thm} There exists a smooth proper rigid-analytic surface $X$ over $\bQ_p$ such that the $E_{{\infty}}$-algebra $\rR\Gamma_\et(X_{\bC_p}, \bQ_\ell)$ is non-formal for every prime $\ell$ (including $\ell=p$). Furthermore, the $E_{{\infty}}$-algebra $\rR\Gamma_\dR(X/\bQ_p)$ is also not formal. 
\end{thm}

A simpler special case of Theorem \ref{thm: wm implies formality intro} consists of smooth proper rigid-analytic spaces with good reduction. They satisfy the weight-monodromy conjecture by the Weil conjectures, and their $\ell$-adic cohomology algebras are formal as a consequence of the fact that Frobenius eigenvalues on different cohomology groups are distinct, cf. \cite[Corollaire 5.3.7]{weil2}.

Theorem \ref{thm: wm implies formality intro} applies to algebraic varieties as well, though it is only of academic interest in this case: even if the variety in question has bad reduction at $p$, we can spread it out and reduce modulo another prime of good reduction to deduce formality of cohomology from the Weil conjectures. This spreading out technique is notably unavailable for rigid-analytic spaces.

As an important technical step towards the proof of Theorem~\ref{thm: wm implies formality intro} for $\ell=p$, we give a description of the derived category of de Rham representations that mimics the definition of the abelian category of de Rham representations. {We denote by $\Rep_{\QQ_p} G_K$ the category of continuous finite-dimensional $\QQ_p$-linear representations of the absolute Galois group $G_K$, and by $\Rep^{\dR}_{\QQ_p} G_K$ its full subcategory of de Rham representations. We write $\cD(\Rep_{\QQ_p} G_K)$ and $\cD(\Rep^\dR_{\QQ_p} G_K)$ for their respective derived $\infty$-categories. Finally, $\cD_{\fil}(G_K;\BB_{\dR})$ denotes an appropriate version of the filtered derived $\infty$-category of condensed $\BB_\dR$-semilinear $G_K$-representations, where $\BB_{\dR}$ is endowed with the usual $\bZ$-indexed filtration $\Fil^i\BB_{\dR}=t^i\BB_{\dR}^+$ (see Notation~\ref{notation:B+dr-reps} for a precise definition).}

\begin{thm}[{{Corollary~\ref{cor:de-rham}}}]\label{thm:de-rham-intro} Let $p$ be a prime number, let $K$ be a finite extension of $\QQ_p$, and let $G_K$ be its absolute Galois group. Then there is an equivalence of symmetric monoidal $\infty$-categories
\[
\cD^b\bigl(\Rep^{\dR}_{\QQ_p} G_K\bigr) \simeq \cD_{\fil}(K)\times_{\cD_{\fil}(G_K;\BB_{\dR})} \cD^b\bigl(\Rep_{\QQ_p} G_K\bigr),
\]
where $\cD_\fil(K)$ is the filtered derived $\infty$-category of $K$-vector spaces.  
\end{thm}

We refer to Corollary~\ref{cor:de-rham} for the precise formulation of this theorem, including the construction of all functors involved in the theorem. Theorem~\ref{thm:de-rham-intro}, combined with the de Rham comparison theorem, allows us to lift $\rR\Gamma_{\et}(X_{\bC_p}, \QQ_p)$ to an $E_\infty$-algebra in $\cD^b\bigl(\Rep^{\dR}_{\QQ_p} G_K\bigr)$, cf. Proposition \ref{prop: qp etale cohomology complex is de rham}. We expect that one can alternatively do this by combining the results of \cite{Deglise-Niziol} with the semi-stable comparison theorem of \cite{Colmez-Niziol-I}, {see Remark \ref{rem: wiesia knows it}.}

We conclude by raising the following question. 

\begin{question}\label{question} Let $K$ be a finite extension of $\QQ_p$, let $\ell$ be a prime number, and let $\cX$ be an admissible formal $\cO_K$-scheme such that its special fiber $\cX_s$ is projective and its adic generic fiber $X\coloneqq \cX_\eta$ is smooth over $K$. Is the $E_\infty$-algebra $\rR\Gamma_\et(X_{\bC_p}, \QQ_\ell)$ formal?
\end{question}

Hansen--Li \cite[Question 1.2]{Hansen-Li} and Scholze \cite[Conjecture 2.4]{scholze-rio} proposed that admitting a formal model with projective reduction could be the correct rigid-analytic analog of the K\"ahler condition in complex geometry. Specifically, they expected that the existence of such a model would imply Hodge symmetry on the adic generic fiber. While the first-named author answered this particular conjecture in the negative \cite{Petrov-counterexample}, one might still ask whether this condition serves as the correct $p$-adic analog from the point of view of other cohomological properties exhibited by K\"ahler manifolds. 

 Question~\ref{question} thus provides a natural further test. Notably, a positive answer to \cite[Conjecture 1.4.7]{Hansen-Zavyalov}, combined with Theorem~\ref{thm: wm implies formality intro}, would affirmatively answer Question~\ref{question}. Consequently, a counterexample to Question~\ref{question} would also disprove the somewhat optimistic \cite[Conjecture 1.4.7]{Hansen-Zavyalov}.

\subsection*{{Idea of the proof.}}
To explain the idea of the proof of Theorem \ref{thm: wm implies formality intro}, recall the following point of view on the formality theorem over $\bC$ provided by the theory of mixed Hodge complexes. For an algebraic variety $X$ over $\bC$, \cite[4.1]{beilinson-absolute-hodge} naturally promotes $\RGamma(X(\bC),\bQ)$ to an object of the derived category $\cD^b(\MHS_{\bQ})$ of the abelian category of mixed Hodge structures. Applying Deligne's fundamental observation that the weight filtration on an $\bR$-mixed Hodge structure comes with a natural splitting, we obtain a decomposition of the underlying complex $\RGamma(X(\bC),\bR)\in \cD(\Vect_{\bR})$ into weight pieces, which is compatible with the algebra structure. When $X$ is smooth proper, $\hH^i(X(\bC),\bR)$ is pure of weight $i$, hence this decomposition is exactly a formality isomorphism.

To prove Theorem \ref{thm: wm implies formality intro}, we show a general formality statement for $E_{\infty}$-algebras in the derived category of Weil--Deligne representations. We denote by $\cD^b_{\pure}(\Rep_{E} WD_K)$ the full subcategory of the bounded derived $\infty$-category of (finite-dimensional) $E$-linear Weil--Deligne representations consisting of objects $M$ such that $\hH^i(M)$ is monodromy-pure of weight $i$, for every $i\in\bZ$. Here $E$ is any coefficient field of characteristic zero.

\begin{thm}[{Corollary~\ref{cor: pure wd algebra is formal}}]\label{thm:intro-pure wd algebra is formal} Let $K$ be a finite extension of $\QQ_p$, let $E$ be a field of characteristic $0$, and let $A$ be an $E_\infty$-algebra in $\cD^b_{\pure}(\Rep_E WD_K)$. Then there exists an equivalence $A\simeq\bigoplus \hH^i(A)[-i]$ of $E_{\infty}$-algebras in $\cD(\Vect_E)$.
\end{thm}

Theorem \ref{thm:intro-pure wd algebra is formal} is proven by showing in Theorem \ref{thm: dbpure is dbpure of mixed} that every object of $\cD^b_{\pure}(\Rep_E\WD_K)$ lifts naturally to an object of the derived category of the abelian category of {\it mixed} Weil--Deligne representations. The desired decomposition then follows from the fact that the weight filtration on a mixed Weil--Deligne representation is naturally split on the level of underlying vector spaces, and the splitting is compatible with tensor products.

The case $\ell\neq p$ in Theorem \ref{thm: wm implies formality intro} now readily follows, using Lemma \ref{lem: repgk derived of its heart} to lift $\RGamma_{\et}(X_{\bC_p},\bQ_{\ell})$ to an $E_{\infty}$-algebra in $\cD^b(\Rep_{\bQ_{\ell}}WD_{K})$.

To prove formality of $\bQ_p$-\'etale cohomology, we apply Fontaine's functor $D_{\pst}$ to the complex $\RGamma_{\et}(X_{\bC_p}, \bQ_p)$, to be able to use the same general result about Weil--Deligne representations. We are in a position to apply the functor $D_{\pst}$ to the \'etale cohomology complex, because we have upgraded it to an object of the derived category $\cD^b(\Rep_{\bQ_p}^{\dR}G_K)$ of the abelian category of de Rham Galois representations, using Theorem \ref{thm:de-rham-intro}, as described above.

Examples in Theorem \ref{thm:main-thm} showing that formality does not always hold for rigid-analytic spaces are constructed by adapting a familiar construction from complex geometry. One of the best-known examples of a compact complex (necessarily non-K\"ahler) manifold whose cohomology is not formal is the Iwasawa manifold \cite[p. 261]{deligne-griffiths-morgan-sullivan}, but its construction does not have a $p$-adic analog.

Instead, we use that primary Kodaira surfaces do have $p$-adic analogs. We show that their \'etale cohomology algebras are equivalent to those of their complex counterparts. In particular, they are not formal because the cohomology algebra contains the cohomology of the Heisenberg group as a tensor factor.

\subsection*{Acknowledgements} We thank Bhargav Bhatt for bringing this problem to our attention and for useful discussions. The idea that the weight-monodromy conjecture might be related to formality is due to him. We are also grateful to Vadim Vologodsky for a very useful discussion, and to Federico Scavia for an instructive example related to Lemma~\ref{lem: repgk derived of its heart}. A.P. was supported by the Clay Research Fellowship. Part of this research was conducted during the authors' participation in the IAS special year on $p$-adic geometry.

We are grateful to the following institutions for their financial support and for providing excellent
working conditions while we completed parts of this project: the School of Mathematics of the Institute for
Advanced Study (A.P. and B.Z.), MIT (A.P.), Princeton University
(B.Z.), and the University of Maryland (B.Z.).  
\section{Recollections on formality}

For a commutative ring $R$, the category of $E_{\infty}$-algebras over $R$ is, by definition, the category of commutative algebra objects in the symmetric monoidal derived $\infty$-category $D(R)$ of complexes of $R$-modules. Given any Grothendieck site $\cC$ with a sheaf $\cO$ of $R$-algebras, the cohomology complex $\RGamma(\cC,\cO)\in D(R)$ can be naturally enhanced with an $E_{\infty}$-algebra structure, using the lax symmetric monoidal structure on the functor $\RGamma(\cC,-)$. 

When $R$ is a $\bQ$-algebra (as is always the case in this paper), the category of $E_{\infty}$-algebras over $R$ is equivalent, via the Dold--Kan correspondence, to the $\infty$-category of commutative differential graded $R$-algebras localized at quasi-isomorphisms; see \cite[Proposition 7.1.4.11]{lurie-ha}. For example, the fact that the singular cohomology complex $\RGamma_{\sing}(X,\bQ)$ of any topological space $X$ can be naturally represented by a commutative differential graded $\bQ$-algebra goes back to Sullivan, cf. \cite[\S 2]{deligne-griffiths-morgan-sullivan}.

We say that an $E_{\infty}$-algebra $A$ is formal if it is equivalent to $\bigoplus\limits_{i\in\bZ} \hH^i(A)[-i]$, where the latter graded commutative algebra is viewed as an $E_{\infty}$-algebra. If $R$ is a $\bQ$-algebra and $A$ is represented by a commutative differential graded algebra $A^{\bullet}=\ldots \xrightarrow{d}A^i\xrightarrow{d} A^{i+1}\xrightarrow{d}\ldots$, then formality is equivalent to the existence of another commutative differential graded algebra $B^{\bullet}$ admitting maps $B^{\bullet}\to A^{\bullet}$ and $B^{\bullet}\to (\hH^{\bullet}(A),0)$ that induce isomorphisms on cohomology, {as one shows using semi-free resolutions \cite[13.4]{drinfeld-dg}}.

\section{Example}

Let $K$ be a finite extension of $\bQ_p$, and denote by $\bC_p = \widehat{\overline{K}}$ the completion of its algebraic closure. We choose any elliptic curve $E$ over $K$ and a line bundle $L$ of degree $1$ on $E$. Consider its total space $Y \coloneqq (\Tot_E L)\setminus E$ with the zero section removed. Note that this is an example of a smooth non-proper algebraic variety whose cohomology is not formal, cf. \cite[Example on p. 203]{morgan}. We will now turn it into a proper rigid-analytic space by taking an appropriate quotient. This smooth variety is equipped with an action of $\bG_m$ given by dilation along the fibers. In particular, $Y$ admits an action of $\bZ$ via a homomorphism $\bZ \to \bG_m$ sending $1$ to $p$.

\begin{construction} We denote by $X \coloneqq Y^\an/p^\bZ$ the rigid-analytic space over $K$ obtained as the quotient of the analytification of $Y$ by the action of $\bZ$ described above.
\end{construction}

By construction, $X \to E^\an$ is a (non-trivial, as we will see) analytic torsor for the constant rigid-analytic group $\bG_m^\an/p^{\bZ}$. The latter group is a Tate elliptic curve, see \cite[Theorem 9.7.3/3]{BGR}\footnote{Strictly speaking, \cite[Theorem 9.7.3/3]{BGR} assumes that $p\neq 2$. However, this is not needed to justify that $\bG_m^\an/p^{\bZ}$ is an elliptic curve.}. In particular, $X$ is a smooth proper (geometrically connected) rigid-analytic surface over $K$. This construction goes back at least to \cite[\S 6 a)]{ueno}, see also \cite[3.4]{wilke} for a slightly different presentation of the same surface.

The main result of this section is that cohomology algebras of $X$ are not formal:

\begin{pr}\label{main}
For all primes $\ell$ (including $\ell=p$) the $E_{\infty}$-algebra $\RGamma_{\et}(X_{\bC_p},\bQ_{\ell})$ is not equivalent to $\bigoplus\limits_{i\geq 0} \hH^i_{\et}(X_{\bC_p},\bQ_{\ell})[-i]$. Also, the de Rham cohomology $E_{\infty}$-algebra $\RGamma_{\dR}(X/K)$ is not equivalent to $\bigoplus\limits_{i\geq 0}\hH^i_{\dR}(X/K)[-i]$.
\end{pr}

We will reduce the computation of \'etale cohomology of this rigid-analytic surface to the computation of the cohomology of its classical complex-analytic analog.

Choose an embedding $K\hookrightarrow \bC$ and denote by $Y_{\bC}^{\an}$ the complex-analytic space associated to $Y_{\bC}$, and let $X_{\bC}$ be the complex manifold obtained as the quotient of $Y_{\bC}^{\an}$ by the same action of $p^{\bZ}$ as above. This compact complex surface is known as a {\it primary Kodaira surface}, cf. \cite[V.5.BIb)]{barth-peters-ven}.

\begin{lm}\label{padic-complex comparison}
There is an equivalence of $E_{\infty}$-algebras $\RGamma_{\et}(X_{\bC_p},\bQ_{\ell})\simeq \RGamma_{\sing}(X_{\bC}(\bC),\bQ_{\ell})$.
\end{lm}
\begin{proof}
As mentioned above, the action of the group $\bZ$ on the algebraic variety $Y$ induces actions on the rigid-analytic space $Y^\an_{\bC_p}$ and the complex analytic space $Y^\an_{\bC}$. The spaces $X_{\bC_p}$ and $X_{\bC}$ are the quotients of $Y^\an_{\bC_p}$ and $Y^\an_{\bC}$ in the respective categories. In particular, we have equivalences of $E_{\infty}$-algebras
\begin{multline}\label{coh as group coh}
\RGamma_{\et}(X_{\bC_p},\bQ_{\ell})\simeq\RGamma\big(\bZ,\RGamma_{\et}(Y^{\an}_{\bC_p},\bQ_{\ell})\big) \\
\RGamma_{\sing}(X_{\bC}(\bC),\bQ_{\ell})\simeq\RGamma\big(\bZ,\RGamma_{\sing}(Y^{\an}_{\bC}(\bC),\bQ_{\ell})\big)
\end{multline}
where the outer $\RGamma$ on the right-hand sides refers to the cohomology of the discrete group $\bZ$.

Combining the comparison between rigid-analytic and algebraic \'etale cohomology (see \cite[Theorem 3.8.1]{Huber-etale}), the Artin comparison between algebraic \'etale and singular cohomology (see \cite[Expos\'e XVI, Th\'eor\`eme 4.1]{SGA4}), and the independence of \'etale cohomology of the algebraically closed ground field (see \cite[Expos\'e XVI, Corollaire 1.6]{SGA4}), we get an equivalence of $E_{\infty}$-algebras in $\cD(\Rep_{\bQ_{\ell}}\bZ)$
\begin{equation*}
\RGamma_{\et}(Y^{\an}_{\bC_p},\bQ_{\ell})\simeq \RGamma_{\et}(Y_{\oK},\bQ_{\ell})\simeq \RGamma_{\sing}(Y^{\an}_{\bC}(\bC),\bQ_{\ell}).
\end{equation*}
Applying the functor $\RGamma(\bZ,-)$ to this equivalence and using the equivalences (\ref{coh as group coh}), we get the result.
\end{proof}

Denote by $\Gamma$ the Heisenberg group \begin{equation*}\Bigg\{\left.\left(\begin{matrix}
1 & a & b \\ 0 & 1 & c \\ 0 & 0 & 1
\end{matrix}\right)\right\vert_{}a,b,c\in\bZ\Bigg\}\subset \rm{GL}_3(\bZ).\end{equation*} It can be characterized as the unique group that can be established as a central extension $1\to \bZ\to \Gamma\xrightarrow{\alpha}\bZ^2\to 1$ of $\bZ^2$ by $\bZ$ such that $\alpha$ induces an isomorphism between the abelianization $\Gamma^{\ab}$ and $\bZ^2$.\footnote{$\hH^2(\bZ^2,\bZ)$ is isomorphic to $\bZ$, under this isomorphism the class of a central extension $0\to\bZ\to G\xrightarrow{\beta}\bZ^2\to 0$ corresponds to the integer $[\tx,\ty]$ for any $\tx,\ty\in G$ lifting the chosen basis elements $x,y\in \bZ^2$. The map $\beta$ induces an isomorphism $G^{\ab}\simeq\bZ^2$ if and only if this integer is $\pm 1$, which is the case for $G=\Gamma$.}

The following description of the topology of Kodaira surfaces is classical, see \cite[(4) on p. 755]{hasegawa} or \cite[p. 374]{borcea}, but let us include a proof for completeness.

\begin{lm}\label{complex kodaira cohomology}
The complex surface $X_{\bC}(\bC)$ is homotopy equivalent to the classifying space of the group $\Gamma\times \bZ$. 
In particular, the $E_{\infty}$-algebra $\RGamma_{\sing}(X_{\bC}(\bC),\bQ_{\ell})$ is equivalent to the group cohomology algebra $\RGamma(\Gamma\times \bZ,\bQ_{\ell})$.
\end{lm}

In the proof below, we will ignore the choice of a base point in the definition of homotopy groups $\pi_i$. Since all spaces considered are connected, this should not cause any confusion. 

\begin{proof}
We will prove that $X_{\bC}(\bC)$ is a $K(\pi,1)$-space whose fundamental group is isomorphic to $\bZ\times\Gamma$, which is equivalent to the assertion of the lemma.

{\it Step~$1$. $Y^{\an}_\bC(\bC)$ is homotopy equivalent to $K(\Gamma, 1)$.} First, we note that the natural map $f\colon Y^{\an}_\bC(\bC)\to E^\an_{\bC}(\bC)$ is a fibration with fibers isomorphic to $\bC^{\times}$. Since both $E^\an_{\bC}(\bC)$ and $\bC^\times$ are $K(\pi, 1)$-spaces, we conclude that $Y^{\an}_{\bC}(\bC)$ is also a $K(\pi, 1)$-space. So we only need to show that $\pi_1(Y^\an_{\bC}(\bC))$ is isomorphic to the Heisenberg group $\Gamma$. 

For this, we consider the homotopy exact sequence for the fibration mentioned above (which is exact on the left because $\pi_2\big(E^\an_\bC(\bC)\big)=0$):
\begin{equation*}
1\to \bZ\to \pi_1\big(Y^\an_\bC(\bC)\big)\to \pi_1\big(E^\an_\bC(\bC)\big)\to 0.
\end{equation*}
We claim that this extension is central. To see this, it suffices to show that the action of $\pi_1\big(E^\an_\bC(\bC)\big)$ on $\bZ \simeq \pi_1(\bC^\times) \simeq \hH_1(\bC^\times, \bZ) \simeq \hH^1_c(\bC^\times, \bZ)$ is trivial. This follows from the fact that the local system $\rR^1 f_!\, \bZ$ is isomorphic to the constant local system $\bZ$.\footnote{To see this claim, note that $f$ admits a compactification by a $\bP^1$-bundle $\overline {f} \colon \bP^\an_E(L\oplus \cO)(\bC) \to E_\bC^\an(\bC)$. Using this, one sees that $\rR^1 f_!\, \bZ \simeq \coker(\overline{f}_*\bZ = \bZ \to \bZ^2) \simeq \bZ$.}

Therefore, we conclude that $\pi_1\big(Y^\an_\bC(\bC)\big)$ fits into a central extension $1\to \bZ \to \pi_1\big(Y^\an_\bC(\bC)\big) \to \bZ^2 \to 1$. It remains only to show that the induced map 
\[
\pi^{\ab}_1\big(Y^\an_\bC(\bC)\big) \to \pi^{\ab}_1\big(E^\an_\bC(\bC)\big)
\]
is an isomorphism. The Hurewicz theorem implies that this is equivalent to showing that the map $\hH_1\big(Y^{\an}_\bC(\bC), \bZ\big) \to \hH_1\big(E^{\an}_\bC(\bC), \bZ\big)$ is an isomorphism. Using the universal coefficient theorem, we see that it suffices to show that the natural map $\hH^1\big(E^\an_\bC(\bC), \bZ\big) \to \hH^1\big(Y^\an_\bC(\bC), \bZ)$ is an isomorphism and both $\hH^2\big(E^\an_\bC(\bC), \bZ\big)$ and $\hH^2\big(Y^\an_\bC(\bC), \bZ)$ are torsion-free. Clearly, $\hH^2\big(E^\an_\bC(\bC), \bZ\big) \simeq \bZ$ is torsion-free. 

For the other claims, we consider the Leray spectral sequence
\[
\mathrm{E}^{p,q}_2 = \hH^p\big(E^{\an}_\bC(\bC), \rR^q f_*\bZ\big) \Rightarrow \hH^{p+q}_{\sing}(Y_\bC^{\an}(\bC), \bZ).
\]
Now we use that $f_*\bZ \simeq \bZ$, $\rR^1f_*\bZ \simeq \bZ$, and $\rR^if_* \bZ =0$ for $i\geq 2$.\footnote{To see these isomorphisms, we also use the $\bP^1$-bundle compactification $\overline{f}$ of $f$ and the associated distinguished triangle $\bZ^{\oplus 2}[-2] \to \rR \overline{f}_* \bZ \to \rR f_* \bZ$.} Hence, the only potentially non-trivial differential $\hH^0\big(E^{\an}_{\bC}(\bC), \rR^1f_*\bZ)\to \hH^2\big(E^{\an}_{\bC}(\bC), f_*\bZ)$ is given, up to a sign, by sending $1$ to the first Chern class $c_1(L)$ of the line bundle $L$ used to define $Y$. Since $\deg L = 1$, we conclude that this map is an isomorphism. This formally implies that $\hH^1\big(E^\an_\bC(\bC), \bZ\big) \to \hH^1\big(Y^\an_\bC(\bC), \bZ)$ is an isomorphism and that 
\[
\hH^2(Y^\an_\bC(\bC), \bZ) \simeq \hH^1(E^{\an}_{\bC}(\bC), \rR^1f_*\bZ) \simeq \hH^1(E^{\an}_{\bC}(\bC), \bZ) \simeq \bZ^2.
\]
This finishes the proof that $Y^{\an}_{\bC}(\bC)$ is homotopy equivalent to $K(\Gamma, 1)$.

{\it Step~$2$. $X_\bC(\bC)$ is homotopy equivalent to $K(\Gamma\times \bZ, 1)$.} Since $X_{\bC}(\bC)$ is the quotient of $Y^{\an}_{\bC}(\bC)$ by a free action of $\bZ$, we conclude that $X_{\bC}(\bC)$ is a $K(\pi, 1)$-space and its fundamental group fits into the following short exact sequence:
\begin{equation*}
1\to \pi_1\big(Y^{\an}_{\bC}(\bC)\big)\to \pi_1\big(X_{\bC}(\bC)\big)\xrightarrow{\pi} \bZ\to 1.
\end{equation*}
This defines an outer action $\rho\colon \bZ \to \mathrm{Out}\big(\pi_1\big(Y_{\bC}^{\an}(\bC)\big)\big)$. Since $\pi$ obviously admits a group-theoretic section, it suffices to show that $\rho$ is trivial to conclude that $\pi_1\big(X_{\bC}(\bC)\big) \simeq \pi_1\big(Y^{\an}_{\bC}(\bC)\big) \times \bZ \simeq \Gamma\times \bZ$. This action is induced by the action of $\bZ$ on $Y_{\bC}^{\rm{an}}(\bC)$ such that $1$ acts by fiberwise multiplication by $p$. This automorphism is homotopic to the identity automorphism, therefore the induced outer action on $\pi_1\big(Y^{\an}_{\bC}(\bC)\big)$ is trivial. This finishes the proof. 
\end{proof}

\begin{proof}[Proof of Proposition \ref{main}]
Combining Lemmas \ref{padic-complex comparison} and \ref{complex kodaira cohomology}, we get that the \'etale cohomology $\RGamma_{\et}(X_{\bC_p},\bQ_{\ell})$ of our rigid-analytic surface is equivalent, even as an $E_{\infty}$-algebra, to the group cohomology $\RGamma(\Gamma\times\bZ,\bQ_{\ell})$. But the latter is non-formal even as an $E_1$-algebra by the classical observation that there exists a non-trivial Massey product in the cohomology of the Heisenberg group $\Gamma$, cf. \cite[p. 261]{deligne-griffiths-morgan-sullivan}. Non-formality of de Rham cohomology now follows by Lemma \ref{dr vs padic etale} below.
\end{proof}

\begin{rem}
One can also compute the cohomology of $X_{\bC_p}$ via purely $p$-adic methods. Assume for simplicity that $E$ has split multiplicative reduction, so that its analytification is a Tate curve $\bG_{m,K}^{\an}/q^{\bZ}$ for some $q\in\fm_K\subset \cO_K$. Then $X_{\bC_p}$ can be described as a quotient of $(\bG_{m,{\bC_p}}^{\an})^2$ by the group $\bZ^2$, this allows one to compute its cohomology via a Leray spectral sequence, see \cite[3.4]{wilke}.
\end{rem}

\section{Derived categories of Galois representations}

\subsection{Representations}

Throughout this subsection, we fix a finite extension $K$ of $\bQ_p$ with absolute Galois group $G_K$ and a prime number $\ell$ (which could be equal to $p$). The main goal of this subsection is to define the correct derived category of continuous $\QQ_\ell$-linear representations of $G_K$. 

For this, we will use the pro\'etale site $BG_{K, \proet}$ as defined in \cite[\textsection 4.3]{proetale}. Its underlying category is the category\footnote{This equivalence follows from \cite[Lemma 5.6.4]{profinitegroups}.} $G_K\text{--pfsets} \simeq \Pro(G_K\text{--fsets})$ of profinite sets with a continuous $G_K$-action with covers given by continuous surjections. 

Any topological space $X$ with a continuous $G_K$-action gives rise to a pro\'etale sheaf $\ud{X}\in \Shv(BG_{K, \proet}; \Set)$ defined via the formula $\ud{X}(S) = \Map_{G_K, \cont}(S, X)$, see \cite[Lemma 4.3.2]{proetale}. In particular, the topological ring $\bQ_{\ell}$ equipped with a trivial action of $G_K$ defines a sheaf of commutative algebras on $BG_{K,\proet}$ that we denote by $\bQ_\ell$ as well. From now on, we denote the abelian category of sheaves of $\bQ_\ell$-modules on $BG_{K,\proet}$ by $\Mod(G_K; \bQ_\ell)$ and its derived $\infty$-category by $\cD(G_K; \bQ_\ell)$.

\begin{construction}\label{construction:condensed-resp} There is a fully faithful, exact, symmetric monoidal functor 
\[
\Rep_{\QQ_\ell} G_K \to \Mod(G_K; \QQ_\ell)
\]
which sends a continuous finite-dimensional $\QQ_\ell$-linear representation $V$ to the associated sheaf $\ud{V}$ on $BG_{K, \proet}$. In what follows, we will {often simply write $V$ in place of $\ud{V}$.} 

Since this functor is exact, it induces a symmetric monoidal functor $\cD(\Rep_{\QQ_\ell} G_K) \to \cD(G_K; \QQ_{\ell})$.
\end{construction}
We will now prove that this functor is fully faithful, using the following lemma. Let $\varphi\in G_K$ be any lift of the arithmetic Frobenius element along the map $G_K\twoheadrightarrow G_K/I_K\simeq\widehat{\bZ}$. 

\begin{lemma}\label{lemma:killing-homs-into-E(1)-galois} Let $V, U\in \Rep_{\QQ_\ell} G_K$ and let $S=\QQ_\ell\cdot e_1 \oplus \QQ_\ell\cdot e_2$ be the two-dimensional continuous $\QQ_\ell$-linear $G_K$-representation such that the inertia $I_K$ acts trivially on $S$, and $\varphi$ acts via the $\QQ_\ell$-linear transformation such that $\varphi(e_1)=e_1$ and $\varphi(e_2)=e_1+e_2$. Let $V \hookrightarrow V\otimes_{\QQ_\ell} \Sym^r S$ be a homomorphism of $G_K$-representations defined via the formula $v\mapsto v\otimes e_1^r$. Then the induced morphism
\[
\Hom_{G_K}\bigl(V\otimes_{\QQ_\ell} \Sym^r S, U\bigr) \to \Hom_{G_K}\bigl(V, U\bigr)
\]
is zero for any $r\geq \dim_{\QQ_\ell} V \cdot \dim_{\QQ_\ell} U$.
\end{lemma}
\begin{proof}
    Note that $\Hom_{G_K}\bigl(V\otimes_{\QQ_\ell} \Sym^r S, U \bigr) \simeq \Hom_{G_K}(\Sym^r S, V^\vee\otimes_{\QQ_\ell} U)$ and $\Hom_{G_K}\bigl(V, U\bigr) \simeq \Hom_{G_K}(\QQ_\ell, V^\vee \otimes_{\QQ_\ell} U)$. Therefore, it suffices to show that the inclusion $\QQ_\ell \to \Sym^r S$, defined via the formula $x\mapsto xe_1^r$, induces the zero morphism $\Hom_{G_K}(\Sym^r S, W) \to \Hom_{G_K}(\QQ_\ell, W)$ for any $W\in \Rep_{\QQ_\ell} G_K$ such that $\dim_{\QQ_\ell} W\leq r$. 

    In other words, we need to show that for any homomorphism $T\colon \Sym^r S\to W$ for $W$ as above, we have $T(e_1^r)=0$. For this, we note that $(\varphi-1)(e_1^r)=0$ and $(\varphi-1)^r(e_2^r)=r!\cdot e_1^r$ in $\Sym^r S$. Therefore, we conclude that for any homomorphism $T\colon \Sym^r S \to W$, we have $T(e_1^r) \in \mathrm{Im}\bigl((\varphi_W-1)^r\bigr)\cap \mathrm{Ker}\bigl(\varphi_W-1\bigr)$. Now note that this intersection is trivial if $r\geq \dim_{\QQ_\ell} W$. This finishes the proof. 
\end{proof}

\begin{lemma}\label{lem: repgk derived of its heart} Let $K$ be a finite extension of $\QQ_p$. Then the functor $\cD^b(\Rep_{\QQ_\ell} G_K) \to \cD(G_K; \QQ_\ell)$ is fully faithful. Its essential image is characterized as the subcategory of objects $M\in \cD(G_K;\bQ_{\ell})$ such that each $\hH^i(M)\in \Mod(G_K;\bQ_{\ell})$ arises from an object of $\Rep_{\bQ_{\ell}}G_K$ and is zero for almost all $i\in\bZ$.
\end{lemma}
This was proven for $\ell=p$ in \cite[Proposition 1.3.2]{Emerton-Kisin}. We give here a uniform argument for all $\ell$.
\begin{proof}
    An inductive argument reduces the question to showing that the natural morphism $\Ext^i_{\Rep_{\QQ_{\ell}} G_K}(V, W) \to \Ext^i_{\Shv(BG_{K, \proet};\bQ_{\ell})}(V, W)$ is an isomorphism for any $V, W\in \Rep_{\QQ_{\ell}} G_K$ and $i\geq 0$. Furthermore, we can replace $V$ with $\QQ_\ell$ and $W$ with $V^{\vee}\otimes W$ and use \cite[Lemma 4.3.9]{proetale} to reduce the question to showing that $\Ext^i_{\Rep_{\QQ_{\ell}} G_K}(\QQ_\ell, W) \to \hH^i_{\cont}(G_K, W)$ is an isomorphism. 
    
    Since this map is an isomorphism for $i=0$, \cite[Proposition~4.2]{Buchsbaum} implies that it suffices to check that for any $V\in \Rep_{\QQ_\ell} G_K$ and $u\in \hH^i_{\cont}(G_K, V)$, there is an injection $V \hookrightarrow V'$ in $\Rep_{\QQ_\ell}G_K$ such that the image of $u$ vanishes in $\hH^i(G_K, V')$. Since the $\ell$-adic cohomological dimension of $G_K$ is $2$, we only need to deal with the cases $i=1$ and $i=2$.
    
    For $i=1$, \cite[Proposition 1.2.6]{Emerton-Kisin} implies that $\Ext^1(\QQ_\ell, V) \to \hH^1_{\cont}(G_K, V)$ is an isomorphism. Thus, the class $u\in \hH^1_{\cont}(G_K, V)$ is represented by an extension $0 \to V \to W \to \QQ_\ell \to 0$. Then the injection $V \hookrightarrow W$ annihilates the class $u\in \hH^1_{\cont}(G_K, V)$.
    
    For $i=2$, local Tate duality implies that $\hH^2_{\cont}(G_K, V)^{\vee} \simeq \hH^0_{\cont}(G_K, V^\vee(1)) \simeq \Hom_{G_K}(\QQ_\ell, V^\vee(1)) \simeq \Hom_{G_K}(V, \QQ_\ell(1))$. Therefore, it suffices to show that, for any $V\in \Rep_{\QQ_\ell} G_K$, we can find an injection $V \hookrightarrow V'$ such that the natural morphism $\Hom_{G_K}\bigl(V', \QQ_\ell(1)\bigr) \to \Hom_{G_K}\bigl(V, \QQ_\ell(1)\bigr)$ is zero. This is achieved by taking $V'\coloneqq  V\otimes_{\QQ_\ell} \Sym^{\dim_{\QQ_\ell} V} S$ from Lemma~\ref{lemma:killing-homs-into-E(1)-galois}.
\end{proof}

\begin{rem}
\begin{enumerate}
    \item 
Note that the analog of Lemma \ref{lem: repgk derived of its heart} is false for an arbitrary finitely generated profinite group $\Gamma$ in place of $G_K$. The category of continuous representations of the group $\Gamma=\SL_2(\bZ_p)$ on finite-dimensional $\bQ_p$-vector spaces is semi-simple, but $\hH^3_{\cont}(\SL_2(\bZ_p),\bQ_p)\simeq \bQ_p$ is non-zero by Lazard's theorem \cite[Th\'eor\`eme V.2.4.10]{lazard}.
\item {We do not know for which other classes of fields Lemma \ref{lem: repgk derived of its heart} holds. Federico Scavia explained to us that it fails, for example, for any field $K$ of characteristic $\neq p$ whose absolute Galois group is isomorphic to $\prod\limits_{n=1}^{\infty}\bZ_p$, and such fields exist.} However, a consequence of Beilinson's basic lemma \cite[Remark 21.5]{bhatt-notes} is that for any finite type scheme $X$ over a field $F$ of characteristic {coprime to} $\ell$ there is a natural object of $\cD^b(\Rep_{\bQ_{\ell}}G_F)$ lifting $\RGamma_{\et}(X_{\oF},\bQ_{\ell})\in \cD(G_F;\bQ_{\ell})$. It seems natural to expect that the same applies to $\ell$-adic cohomology of any smooth proper rigid-analytic space $X$ over a non-archimedean field $K$ of characteristic coprime to $\ell$.
\end{enumerate}
\end{rem}

\subsection{Derived category of de Rham representations.} 

Recall that a finite-dimensional $\bQ_p$-representation $V$ is called {\it de Rham} if there exists a $K$-vector space $D$ together with a $G_K$-equivariant isomorphism
\begin{equation}\label{bdrpuls comparison iso filtered}
V\otimes_{\bQ_p}\BB_{\dR}\simeq D\otimes_K \BB_{\dR}
\end{equation}
of $\BB_{\dR}$-vector spaces, {compatible with the filtrations}. Alternatively, we may rephrase this condition by saying that there exists a filtered $K$-vector space $D$ together with an isomorphism
\begin{equation}\label{bdrplus comparison iso}
V\otimes_{\bQ_p}\BB_{\dR}^+\simeq \Fil^0(D\otimes_K \BB_{\dR})
\end{equation}
where $\Fil^0$ on the right-hand side is taken with respect to the tensor product of the given filtration on $D$ with the filtration $\Fil^n\BB_{\dR}\coloneqq t^n\BB_{\dR}^+$ on $\BB_{\dR}$. The filtered vector space $D$ is in fact naturally recovered from the de Rham representation $V$. 

In this subsection, we prove that the verbatim derived analog of the condition (\ref{bdrplus comparison iso}) recovers the correct derived category of de Rham representations, see Theorem~\ref{thm: de rham derived category pullback}. 

\begin{notation}\label{notation:B+dr-reps}
Topological $\bQ_p$-algebras $\BB_{\dR}^+$ and $\BB_{\dR}$ equipped with a continuous action of $G_K$ give rise to commutative algebras in the category $\Mod(G_K;\bQ_p)$, and we denote by $\cD(G_K;\BB_{\dR}^+)$ the symmetric monoidal $\infty$-category of modules over the corresponding commutative algebra object $\BB_{\dR}^+\in \cD(G_K;\bQ_p)$. We have a functor $V\mapsto V\otimes_{\bQ_p}\BB_{\dR}^+ \colon \cD(G_K;\bQ_p)\to \cD(G_K;\BB_{\dR}^+)$.
\end{notation}

For an $\infty$-category $\cC$, we denote by
\[
\cC_{\fil}\coloneqq \Fun\bigl((\ZZ, \leq)^{\op},\cC\bigr)
\]
the corresponding filtered $\infty$-category, where our convention is to use decreasing filtrations. If $\cC$ is equipped with a symmetric monoidal structure, $\cC_{\fil}$ can be endowed with the Day convolution symmetric monoidal structure (see \cite[Construction 2.2.6.7]{lurie-ha}). Explicitly, we have $\Fil^n(X\otimes Y)=\colim\limits_{i+j\geq n}\Fil^iX\otimes\Fil^j Y$ for two objects $X,Y\in \cC_{\fil}$. The assignment $\Fil^\bullet C \mapsto \Fil^0 C$ can be promoted to a lax symmetric monoidal functor $\Fil^0\colon \cC_\fil \to \cC$, see \cite[Corollary 3.8]{Ben-Moshe-Schlank}. 

\begin{notation}\label{notation:trivial-filtration} For $C\in \cC$, we define \emph{$C$ with the trivial filtration} to be the filtered object 
\[
\Fil^n C = \begin{cases} C & \text{ if } n\leq 0 \\
0 & \text{ if } n> 0.
\end{cases}
\]
By a slight abuse of notation, we often denote this filtered object simply by $C$. This can be promoted to a symmetric monoidal functor $\mathrm{triv}\colon \cC \to \cC_{\fil}$ which is a left adjoint to $\Fil^0$ (see \cite[Corollary 3.6]{Ben-Moshe-Schlank}).
\end{notation}

For brevity, we denote by $\cD_{\fil}(K)\coloneqq \cD(\Vect_K)_{\fil}$ the filtered derived $\infty$-category of $K$-vector spaces.

\begin{lemma}\label{rhom in dfil} For an object $D\in \cD_{\fil}(K)$ there is a natural equivalence $\RHom_{\cD_{\fil}(K)}(K,D)\simeq \Fil^0 D$, where $K\in\cD_{\fil}(K)$ is equipped with the trivial filtration from Notation~\ref{notation:trivial-filtration}.
\end{lemma}
\begin{proof}
    This follows immediately from the definition of $\mathrm{triv} \colon \cD(\Vect_K) \to \cD_\fil(K)$ as the left adjoint to $\Fil^0\colon \cD_\fil(K) \to \cD(\Vect_K)$. 
\end{proof}

\begin{construction}\label{fil0Bdr construction} We now formally define the symmetric monoidal functor \[D\mapsto\Fil^0(D\otimes_{K} \BB_{\dR})\colon \cD_{\fil}(K)\to \cD(G_K;\BB_{\dR}^+)\]
which is a derived analog of the functor appearing on the RHS of (\ref{bdrplus comparison iso}).

The $\bZ$-indexed filtration $\Fil^n\BB_{\dR}\coloneqq t^n\BB_{\dR}^+$ on the commutative algebra object $\BB_{\dR}$ of the abelian symmetric monoidal category $\Mod(G_K;\bQ_p)$ makes $\BB_{\dR}$ into a commutative algebra object of the filtered derived $\infty$-category $\cD_{\fil}(G_K;\bQ_p)\coloneqq \cD(G_K;\bQ_p)_{\fil}$. We can then consider the composition
\begin{multline}\label{eqn: fil0Bdr defining composition}
\cD_{\fil}(K)\xrightarrow{}\Mod_{K^{\delta}}\cD_{\fil}(G_K;\bZ)\xrightarrow{D\mapsto D\otimes_{K^{\delta}}\BB_{\dR}}\Mod_{\BB_{\dR}}\cD_{\fil}(G_K;\bQ_p)\\ \xrightarrow{\Fil^0}\Mod_{\Fil^0\BB_{\dR}}\cD(G_K;\bQ_p)=\cD(G_K;\BB_{\dR}^+)
\end{multline}
where the first functor sends an object $D\in \cD_{\fil}(K)$ to the corresponding object of $\Mod_{K^{\delta}}\cD_{\fil}(G_K;\bZ)$, where $K^{\delta}$ denotes $K$ viewed as a commutative algebra object of $\cD_{\fil}(G_K;\bZ)$ endowed with trivial filtration and Galois action, and with {\it discrete} topology. 

The first two functors in (\ref{eqn: fil0Bdr defining composition}) are symmetric monoidal, and the last functor $\Fil^0$ has a structure of a lax symmetric monoidal functor. We reinterpret $\BB_{\dR}^+$-modules as filtered $\BB_{\dR}$-modules satisfying an additional property, to prove that the composition (\ref{eqn: fil0Bdr defining composition}) is strictly symmetric monoidal:
\end{construction}

\begin{lm}\label{lemma:filtered-bdr-bdr+}
Consider the full $\infty$-subcategory $\cL$ of $\cD_{\fil}(G_K;\BB_{\dR})$ consisting of filtered objects $M$ such that the natural map $\Fil^0M\otimes_{\Fil^0\BB_{\dR}}\Fil^n\BB_{\dR}\to \Fil^nM$ is an equivalence, for every $n\in\bZ$. Then the composition $\cL\hookrightarrow\cD_{\fil}(G_K;\BB_{\dR})\xrightarrow{\Fil^0}\cD(G_K;\BB_{\dR}^+)$ is an equivalence of symmetric monoidal $\infty$-categories.

Moreover, the functor $\cD_{\fil}(K)\xrightarrow{D\mapsto D\otimes_K \BB_{\dR}}\Mod_{\BB_{\dR}}\cD_{\fil}(G_K;\bQ_p)$ lands inside $\cL$, and in particular the composition (\ref{eqn: fil0Bdr defining composition}) is strictly symmetric monoidal.
\end{lm}

\begin{proof}
The map of commutative algebras $\BB_{\dR}^+\to \BB_{\dR}$ in $\Mod(G_K;\bQ_p)$ can be upgraded to a filtered map, where $\BB_{\dR}^+$ is equipped with a trivial filtration, and $\BB_{\dR}$ is equipped with the standard $t$-adic filtration. Scalar extension along this map defines a symmetric monoidal functor $\cD(G_K;\BB_{\dR}^+)\to \cD_{\fil}(G_K;\BB_{\dR})$. It lands inside the subcategory $\cL$, and one readily sees that it is inverse to the functor $\Fil^0\colon \cL\to \cD(G_K;\BB_{\dR}^+)$.

Every object $M$ of the form $D\otimes_K \BB_{\dR}$ with $D\in \cD_{\fil}(K)$ is contained in $\cL$, because the map $\Fil^0M\otimes_{\BB_{\dR}^+}t^n\BB_{\dR}^+\to \Fil^nM$ can be presented as a map of colimits
\[
\colim\limits_{i+j\geq 0}(\Fil^iD\otimes_K t^{j}\BB_{\dR}^+\otimes_{\BB_{\dR}^+}t^n \BB^+_{\dR})\to \colim\limits_{i+j\geq 0}\Fil^iD\otimes_K t^{n+j}\BB_{\dR}^+
\]
defining the Day convolution product, and this map is an equivalence already on each term of the colimit.

To check that the lax symmetric monoidal functor $D\mapsto \Fil^0(D\otimes_K \BB_{\dR})\colon \cD_{\fil}(K)\to \cD(G_K;\BB_{\dR}^+)$ is strictly symmetric monoidal we may compose it with the fully faithful symmetric monoidal embedding $\cD(G_K;\BB_{\dR}^+)\simeq\cL\hookrightarrow\cD_{\fil}(G_K;\BB_{\dR})$. This composition is simply the functor $D\mapsto D\otimes_K \BB_{\dR}$ which is strictly symmetric monoidal by construction.
\end{proof}

Before we discuss the description of $\cD^b(\Rep^{\dR}_{\bQ_p} G_K)$, we need the following two lemmas. 

\begin{lemma}\label{lemma:heart-t-structure} Let $(V, D, \alpha)$ be a triple of a finite-dimensional continuous $\bQ_p$-linear representation $V\in \Rep_{\bQ_p} G_K$, a filtered complex $D\in \cD_\fil(K)$, and an isomorphism $\Fil^0(D\otimes_K \BB_\dR) \simeq V \otimes_{\QQ_p} \BB^+_\dR$. Then, for any integer $i\in \ZZ$, the object $\Fil^i D$ is concentrated in degree $0$ and the transition map $\Fil^{i+1} D \to \Fil^i D$ is injective.
\end{lemma}
\begin{proof}
We have that
\[
V\otimes_{\bQ_p}\bC_p\simeq\Fil^0(D\otimes_K \BB_{\dR})/t\simeq \gr^0(D\otimes_K \BB_{\dR})\simeq \bigoplus\limits_{i\in\bZ} \gr^i D\otimes_K \bC_p(-i)
\]
is concentrated in degree $0$, where the second isomorphism follows from Lemma~\ref{lemma:filtered-bdr-bdr+}. Therefore, each $\gr^i D$ is concentrated in degree $0$, and almost all of these graded quotients are zero. In particular, objects $\Fil^iD$ stabilize as $i\to  \infty$ and as $i\to -\infty$. Assume that $\Fil^d D = \Fil^{d+1} D = \dots $ and put $H\coloneqq \Fil^d D$ to be this stable value. 

Define the object $\ud{H} \in \cD_\fil(K)$ such that $\Fil^i \ud{H}=H$ for any $i\in \ZZ$ and all transition maps are the identity maps. Then there is a natural morphism $\ud{H} \to D$ and we set $D'\coloneqq \mathrm{Cone}(\ud{H} \to D)$. 

{\it Step~$1$. We show that each $\gr^i D'$ is a finite-dimensional vector space concentrated in degree $0$ and only finitely many of $\gr^i D'$ are non-zero.} We note that $\gr^i \ud{H}\simeq 0$ for any $i\in \ZZ$. Thus, we conclude that $\gr^i D \simeq \gr^i D'$. Finally, the first paragraph of the proof showed that $\gr^i D$ are finite-dimensional $K$-vector spaces and all but finitely many of them vanish. 

{\it Step~$2$. We show that each $\Fil^i D'$ is a vector space concentrated in degree $0$ and the transition maps $\Fil^{i+1} D' \to \Fil^{i} D'$ are injective.} Using Step~$1$, it suffices to show that $\lim\limits_{n} \Fil^n D'=0$. This follows from the following sequence of equalities 
\[
\lim_n \Fil^n D' \simeq \mathrm{Cone}\bigl(\lim_n\Fil^n\ud{H} \to \lim_n \Fil^n D\bigr) \simeq \mathrm{Cone}(H\to H)\simeq 0. 
\]

{\it Step~$3$. We show that $H\simeq 0$.}  First, we note that Step~$2$ and the fact that there are no higher Ext groups in $\Vect_K$ imply that the filtration on $D'$ splits non-canonically. This implies that we have a non-canonical isomorphism
\[
\Fil^0(D'\otimes_K \BB_\dR) \cong \bigoplus_{i\in \ZZ}\gr^i D' \otimes_K t^{-i} \BB^+_\dR.
\]
Furthermore, we observe that $\Fil^0(\ud{H}\otimes_K \BB_\dR) \simeq H\otimes_K \BB_\dR$ and we have $\Fil^0(D\otimes_K B_\dR)\simeq V\otimes_{\bQ_p} \BB_\dR^+$. Thus, after applying the functor $\Fil^0(-\otimes_K \BB_\dR)$ to the exact triangle $\ud{H} \to D \to D'$, we get the following fiber sequence 
\[
H\otimes_K \BB_\dR \to V\otimes_{\bQ_p} \BB_\dR^+ \to \bigoplus_{i\in \ZZ}\gr^i D' \otimes_K t^{-i} \BB^+_\dR
\]
in $\cD(\BB_\dR^+)$. Now the assumption that $V$ is finite-dimensional and Step~$1$ guarantee that the last two terms in the above exact triangle are concentrated in degree $0$ and are finitely generated over $\BB_\dR^+$. This implies that $H\otimes_K \BB_\dR$ lies in $\cD_{\mathrm{coh}}(\BB_\dR^+)$. However, this is possible only when $H=0$. 

Finally, the conclusions of Steps~$2$ and $3$ finish the proof of the lemma. 
\end{proof}

\begin{lemma}\label{lemma:exts-are-universal} Let $\cA$ be an abelian category. Then the sequence $\{\Ext^i_{\cA}(X, -) \coloneqq \Hom_{D(\cA)}(X, -[i])\colon \cA \to \mathrm{Ab} \}_{i\in \ZZ_{\geq 0}}$ is a universal $\delta$-functor.
\end{lemma}
We emphasize that we do not assume that $\cA$ has enough injective or projective objects. This assumption would be inadequate for our applications. 
\begin{proof}
    First, \cite[\href{https://stacks.math.columbia.edu/tag/06XP}{Tag 06XP}]{stacks-project} ensures that $\{\Ext^i_{\cA}(X, -)\}_{i\in \ZZ_{\geq 0}}$ can indeed be promoted to a $\delta$-functor. Then the dual version of \cite[Proposition~4.2]{Buchsbaum} implies that it suffices to show that, for all $i>0$, for any $Y\in \cA$ and $u\in \Ext^i_\cA(X, Y)$, there is an injection $Y\hookrightarrow Y'$ such that the image of $u$ in $\Ext^i_{\cA}(X, Y')$ vanishes. Now \cite[\href{https://stacks.math.columbia.edu/tag/06XU}{Tag 06XU}]{stacks-project} identifies $\Ext^i_\cA(X, Y)$ with the Yoneda Ext group. Therefore, the class $u$ is represented by an exact sequence $0 \to Y \to Z_{i-1} \to \dots{} \to Z_0 \to X \to 0$. One readily checks that the image of $u$ in $\Ext^i_\cA(X, Z_{i-1})$ vanishes. 
\end{proof}

Now we are ready to prove the main result of this section.

\begin{thm}\label{thm: de rham derived category pullback}
The derived $\infty$-category $\cD^b(\Rep^{\dR}_{\bQ_p}G_K)$ is equivalent, as a symmetric monoidal $\infty$-category, to the fiber product of the diagram of $\infty$-categories
\begin{equation}\label{repdr pullback diagram}
\begin{tikzcd}
\cD^b(\Rep_{\bQ_p}G_K)\arrow[rd, "{V\mapsto V\otimes_{\bQ_p}\BB_{\dR}^+}", swap] & & \cD_{\fil}(K)\arrow[ld, "{D\mapsto \Fil^0(D\otimes_K \BB_{\dR})}"] \\
& \cD(G_K;\BB_{\dR}^+)
\end{tikzcd}
\end{equation}
Here, the left functor is the composition of the fully faithful functor from Lemma \ref{lem: repgk derived of its heart} with the base change functor $\cD(G_K;\bQ_p)\to \cD(G_K;\BB_{\dR}^+)$, and the right functor was introduced in Construction \ref{fil0Bdr construction}.
\end{thm}
\begin{rem}
Our proof is specific to $K$ being finite over $\bQ_p$, as it uses Tate duality. We do not know whether this equivalence holds for an arbitrary $p$-adic field $K$ with a perfect residue field.
\end{rem}


Before proving Theorem~\ref{thm: de rham derived category pullback}, we state an immediate consequence that provides a derived analog of the standard definition of de Rham representations. Recall that a representation $V\in \Rep_{\QQ_p} G_K$ is de Rham if and only if there exists a filtered $K$-vector space $D$ and a $G_K$-equivariant filtered isomorphism $D\otimes_K \BB_\dR \simeq V\otimes_{\QQ_p} \BB_\dR$.

\begin{cor}\label{cor:de-rham} The derived $\infty$-category $\cD^b(\Rep^{\dR}_{\bQ_p}G_K)$ is equivalent, as a symmetric monoidal $\infty$-category, to the fiber product of the diagram of $\infty$-categories
\begin{equation*}
\begin{tikzcd}
\cD^b(\Rep_{\bQ_p}G_K)\arrow[rd, "{V\mapsto V\otimes_{\bQ_p}\BB_{\dR}}", swap] & & \cD_{\fil}(K)\arrow[ld, "{D\mapsto D\otimes_K \BB_{\dR}}"] \\
& \cD_{\fil}(G_K;\BB_{\dR})
\end{tikzcd}
\end{equation*}
\end{cor}
\begin{proof}
    This follows immediately from Theorem~\ref{thm: de rham derived category pullback} and Lemma~\ref{lemma:filtered-bdr-bdr+}.
\end{proof}

\begin{proof}[{Proof of Theorem~\ref{thm: de rham derived category pullback}}]
Denote by $\cD_K$ the fiber product of (\ref{repdr pullback diagram}), which is a symmetric monoidal stable $\infty$-category. We denote objects of $\cD_K$ by triples $(V,D, \alpha)$ with $V\in \cD^b(\Rep_{\bQ_p}G_K)$, $D\in\cD_{\fil}(K)$, and an equivalence $\alpha \colon\Fil^0(D\otimes_{K}\BB_{\dR}) \xrightarrow{\sim} V\otimes_{\bQ_p}\BB_{\dR}^+$. We have a functor $F\colon \Rep^{\dR}_{\bQ_p}G_K\to \cD_K$
sending a de Rham representation $W$ to the triple $(W, D_{\dR}(W), \alpha)$, where $\alpha \colon \Fil^0\bigl(D_\dR(W) \otimes_K \BB_\dR\bigr) \xrightarrow{\sim} W\otimes_{\QQ_p} \BB^+_\dR$ is the natural isomorphism. Since the functor $D_{\dR}$ is filtered exact and symmetric monoidal \cite[Proposition 1.5.2]{fontaine-semistables}, the functor $F$ takes short exact sequences of de Rham representations to fiber sequences in $\cD_K$ and induces a symmetric monoidal functor $\cD^b(\Rep^{\dR}_{\bQ_p}G_K)\to \cD_K$.

{\it Equivalence on the hearts.} The three categories $\cD^{b}(\Rep_{\bQ_p} G_K),\cD(G_K;\BB_{\dR}^+),\cD_{\fil}(K)$ appearing in (\ref{repdr pullback diagram}) are endowed with natural $t$-structures, using that they are the derived categories of the abelian categories $\Rep_{\bQ_p}G_K, \Mod(G_K;\BB_{\dR}^+), \Fun\bigl((\bZ, \leq)^{\op},\Vect_K\bigr)$, respectively. Both functors in (\ref{repdr pullback diagram}) are $t$-exact, hence these $t$-structures induce a $t$-structure on the category $\cD_K$. By definition, an object $(V, D, \alpha)\in \cD_K$ lies in its heart if $V$ and the objects $\Fil^nD$ for each $n\in\bZ$ are concentrated in degree zero. Lemma~\ref{lemma:heart-t-structure} implies that this is equivalent to just requiring that $V$ is concentrated in degree zero. Furthermore, in this case, all the transition maps $\Fil^{i+1} D \to \Fil^i D$ are injective. In other words, the heart of $\cD_K$ consists of triples $(V,D,\alpha)$ where $V$ is a Galois representation, $D$ is a filtered $K$-vector space, and $\alpha$ is a Galois-equivariant isomorphism $V\otimes_{\bQ_p}B_{\dR}^+\simeq \Fil^0(D\otimes B_{\dR})$. The existence of such an isomorphism implies that $V$ is de Rham, and $D$ together with $\alpha$ is uniquely recovered from $V$ as $D_{\dR}(V)$.
Thus, we conclude that the functor $\cD^b(\Rep^{\dR}_{\bQ_p}{G_K})\to \cD_K$ induces an equivalence of $\Rep^{\dR}_{\bQ_p}G_K$ with the heart of $\cD_K$.

{\it Full faithfulness of the functor $\cD^b(\Rep^{\dR}_{\bQ_p}G_K)\to \cD_K$.} It suffices to show that for any two de Rham representations $V_1,V_2\in \Rep^{\dR}_{\bQ_p}G_K$ the induced map 
\[
\RHom_{\Rep^{\dR}_{\bQ_p}G_K}\bigl(V_1,V_2\bigr)\to \RHom_{\cD_K}\Bigl(\bigl(V_1,D_{\dR}(V_1), \alpha_1\bigr), \bigl(V_2,D_{\dR}(V_2), \alpha_2\bigr)\Bigr)
\]
is an equivalence. Using the symmetric monoidal structure, it suffices to treat the case $V_1=\bQ_p$, and we denote $V_2$ by $V$ for brevity.

By definition, $\RGamma(\cD_K,V)\coloneqq \RHom_{\cD_K}\Bigl(\bigl(\bQ_p, K, \alpha_{\QQ_p}\bigr),\bigl(V,D_{\dR}(V), \alpha_V\bigr)\Bigr)$ fits into the cartesian diagram
\begin{equation*}
\begin{tikzcd}
\RGamma(\cD_K,V)\arrow[r]\arrow[d]\arrow[dr, phantom, "\lrcorner", very near start]  &\RHom_{\cD_{\fil}(K)}\bigl(K,D_{\dR}(V)\bigr)\arrow[d] \\
\RHom_{\cD(G_K;\bQ_p)}(\bQ_p, V)\arrow[r] & \RHom_{\cD(G_K;\BB_{\dR}^+)}\bigl(\BB_{\dR}^+, V\otimes_{\QQ_p} \BB_{\dR}^+\bigr),
\end{tikzcd}
\end{equation*}
where we have appealed to the full faithfulness from Lemma \ref{lem: repgk derived of its heart} to be able to compute the bottom left object as $\RHom$ in $\cD(G_K;\bQ_p)$. More explicitly, the corners of this cartesian diagram can be rewritten as

\begin{equation}\label{eqn:fiber-square}
\begin{tikzcd}
  \RGamma(\cD_K,V) \arrow[r] \arrow[d] \arrow[dr, phantom, "\lrcorner", very near start] & \Fil^0D_{\dR}(V) \arrow[d] \\
  \RGamma_{\cont}(G_K,V) \arrow[r] & \RGamma_{\cont}\bigl(G_K,V\otimes_{\QQ_p} \BB_{\dR}^+\bigr)
\end{tikzcd}
\end{equation}
where we used Lemma \ref{rhom in dfil} for the top right corner, and \cite[Lemma 4.3.9(1)]{proetale} to identify the bottom terms with continuous group cohomology complexes. The right vertical map identifies $\Fil^0D_{\dR}(V)$ with $\hH^0_{\cont}(G_K,V\otimes_{\QQ_p} \BB_{\dR}^+)$, by definition of the filtration on $D_{\dR}(V)$. Hence, the cohomology groups $\hH^i(\cD_K,V)\coloneqq \hH^i(\RGamma(\cD_K,V))$ fit into a long exact sequence

\begin{equation}\label{coh in ck long exact seq}
\begin{tikzcd}[column sep=1.5em, row sep=2.5em]
0 \arrow[r] 
  & \hH^0(\cD_K, V) \arrow[r] 
  & \hH^0_{\cont}(G_K, V) \arrow[r] 
  & 0 \arrow[dll, out=0, in=180, looseness=2, overlay] \\
  & \hH^1(\cD_K, V) \arrow[r] 
  & \hH^1_{\cont}(G_K, V) \arrow[r] 
  & \hH^1_{\cont}(G_K, V \otimes_{\QQ_p} \BB_{\dR}^+) \arrow[dll, out=0, in=180, looseness=2, overlay] \\
  & \hH^2(\cD_K, V) \arrow[r] 
  & \hH^2_{\cont}(G_K, V) \arrow[r] 
  & \hH^2_{\cont}(G_K, V \otimes_{\QQ_p} \BB_{\dR}^+) \arrow[r] 
  & \ldots
\end{tikzcd}
\end{equation}

Cohomology of $G_K$ with coefficients in $V\otimes \BB_{\dR}^+$ vanishes in degrees $\geq 2$ by \cite[Proposition 3.3(iii)]{fontaine-presque}, and cohomology with coefficients in $V$ vanishes in degrees $\geq 3$. Hence, the last displayed term, as well as all omitted terms in the sequence (\ref{coh in ck long exact seq}) are zero.

Since the map $\RHom_{\Rep^{\dR}_{\bQ_p}G_K}(\bQ_p, V)\to \RGamma(\cD_K,V)$ induces an isomorphism on $\hH^0$, \cite[Proposition~4.2]{Buchsbaum} implies that it suffices to check that for any $V\in \Rep^\dR_{\QQ_p} G_K$ and $u\in \hH^i(\cD_K, V)$, there is an injection $V \hookrightarrow V'$ in $\Rep^\dR_{\QQ_p}G_K$ such that the image of $u$ vanishes in $\hH^i(\cD_K, V')$.

For $i=1$, this amounts to the fact that 
\begin{multline*}
\hH^1(\cD_K,V)=\ker\bigl(\hH^1_\cont(G_K,V)\to \hH^1_\cont(G_K,V\otimes_{\QQ_p} \BB_{\dR}^+)\bigr)=\\= \ker\bigl(\hH^1_\cont(G_K,V)\to \hH^1_\cont(G_K,V\otimes_{\QQ_p} \BB_{\dR})\bigr)
\end{multline*}
classifies extensions $0\to V\to W\to\bQ_p\to 0$ in which $W$ is de Rham. The class $[W]\in \hH^1_\cont(G_K, V)$ then dies after the injection $V\hookrightarrow W$.

For $i=2$, the observations that diagram~(\ref{eqn:fiber-square}) is cartesian and that $\Fil^0D_\dR(V)$ is concentrated in degree $0$ imply that $\hH^2(\cD_K,V) \simeq \hH^1_\cont\bigl(G_K, V\otimes_{\QQ_p} (\BB_\dR^+/\QQ_p)\bigr)$. This functor is described as $\bigl(D_{\cris}(V^{\vee}(1))^{\varphi=1}\bigr)^{\vee}$ in Lemma \ref{lem: h2 in Ck description} below. We will prove that already the functor $V\mapsto D_{\cris}(V^{\vee}(1))^{\vee}$ is effaceable on $\Rep^{\dR}_{\bQ_p}G_K$. Dualizing, this is equivalent to showing that for any de Rham representation $V$ there exists a surjection of de Rham representations $W\twoheadrightarrow V$ such that the map $D_{\cris}(W)\to D_{\cris}(V)$ is zero. We denote by $K_0\subset K$ the maximal subfield unramified over $\bQ_p$.

\begin{lm}\label{lem: Dst monoidal}
If $V_1, V_2$ are two $\bQ_p$-representations of $G_K$ such that $V_2$ is semi-stable, then the natural map $D_{\stab}(V_1)\otimes_{K_0}D_{\stab}(V_2)\to D_{\stab}(V_1\otimes_{\bQ_p}V_2)$ is an isomorphism.
\end{lm}
\begin{proof}
By definition of $D_{\stab}$, we have $D_{\stab}(V_1\otimes_{\bQ_p}V_2)=(V_1\otimes_{\bQ_p}V_2\otimes_{\bQ_p} \BB_{\stab})^{G_K}$, where the action of $G_K$ is the diagonal one on all three tensor factors. Using that $V_2$ is semi-stable, we identify $V_2\otimes_{\bQ_p} \BB_{\stab}$ with $D_{\stab}(V_2)\otimes_{K_0}\BB_{\stab}$ where the Galois action is on the second factor only. Hence we have
\[
\begin{aligned}
(V_1\otimes_{\bQ_p}V_2\otimes_{\bQ_p} \BB_{\stab})^{G_K} &\simeq (V_1\otimes_{\bQ_p}D_{\stab}(V_2)\otimes_{K_0}\BB_{\stab})^{G_K} \\
&\simeq (V_1\otimes_{\bQ_p}\BB_{\stab})^{G_K}\otimes_{K_0} D_{\stab}(V_2) \simeq D_{\stab}(V_1)\otimes_{K_0}D_{\stab}(V_2)
\end{aligned}
\]
In particular, the source and target of the natural map $D_{\stab}(V_1)\otimes_{K_0}D_{\stab}(V_2)\to D_{\stab}(V_1\otimes_{\bQ_p}V_2)$ have equal dimensions. But this map is injective for any pair of representations $V_1,V_2$ (\hspace{1sp}\cite[1.5]{fontaine-semistables}), hence it is an isomorphism in our case.
\end{proof}

Let $T$ be the $2$-dimensional representation fitting into a short exact sequence
\[
0\to\bQ_p(1)\to T\to\bQ_p\to 0
\]
whose class in $\hH^1_\cont(G_K,\bQ_p(1))$ is equal to the image of $p\in K^{\times}$ under the Kummer map. The representation $T$ is semi-stable with non-zero monodromy operator $N$ on $D_{\stab}(T)$. Let $r$ be an integer such that the $r$-th power $N^r_V$ of the monodromy operator $N_V$ on $D_{\stab}(V)$ vanishes. We claim that the surjection $W\coloneqq \Sym^r T\otimes_{\QQ_p} V\to V$ induces the zero map on $D_{\cris}$.

Lemma \ref{lem: Dst monoidal} implies that we have a sequence of isomorphisms
\[
D_{\stab}(W)=D_{\stab}(\Sym^r T\otimes_{\QQ_p} V) \simeq D_{\stab}(\Sym^r T)\otimes_{K_0}D_{\stab}(V)\simeq K_0[x]/x^{r+1}\otimes_{K_0} D_{\stab}(V),
\]
where the monodromy operator on $D_{\stab}(W) \simeq K_0[x]/x^{r+1}\otimes_{K_0} D_{\stab}(V)$ is given by $N_W=x\otimes 1+1\otimes N_V$ and the map $D_{\stab}(W)\simeq K_0[x]/x^{r+1}\otimes_{K_0} D_{\stab}(V) \to D_{\stab}(V)$ is equal to the evaluation at $x=0$. The space $D_{\cris}(W)$ is, by definition, the kernel of $N_W$ on $D_{\stab}(W)$ which can be identified with the space of elements of the form $\sum\limits_{i=0}^r x^i\otimes v_i$ for $v_i\in D_{\stab}(V)$ such that $v_i=-N_V(v_{i+1})$ for all $i=0, \dots, r-1$ and $N_V(v_0)=0$. In particular, the vector $v_0=(-1)^r N_V^r(v_r)$ vanishes by our assumption on the order of nilpotence of $N_V$. That is, the map $D_{\stab}(W)^{N_W=0}\to D_{\stab}(V)^{N_V=0}$ is zero, as desired.

This finishes the proof of effaceability of $\hH^2(\cD_K,V)$ and concludes the proof of full faithfulness. Together with the equivalence on the hearts, this proves the theorem, because every object of $\cD_K$ is a successive extension of shifts of objects from the heart.
\end{proof}

\begin{lm}\label{lem: h2 in Ck description}
For any finite-dimensional representation $V\in \Rep_{\bQ_p}G_K$ there is a natural isomorphism
\[
\hH^1_\cont\bigl(G_K,V\otimes_{\bQ_p} (\BB_{\dR}^+/\bQ_p)\bigr)\simeq \bigl(D_{\cris}\bigl(V^{\vee}(1)\bigr)^{\varphi=1}\bigr)^{\vee}.
\]
\end{lm}
\begin{proof}
This is a special case of Tate duality for Galois cohomology with coefficients in Banach--Colmez spaces, proved in \cite[Th\'eor\`eme 6.1]{fontaine-presque}.

First, \cite[Th\'eor\`eme 2.14(iii)]{fontaine-presque} ensures that $\hH^1_\cont(G_K,V\otimes_{\bQ_p} t^n\BB_{\dR}^+)=0$ for any sufficiently large $n$ (it suffices to take $n$ to be larger than all {integral} generalized Hodge--Tate weights of $V$). The natural map $\hH^1_\cont\bigl(G_K,V\otimes_{\bQ_p} (\BB_{\dR}^+/\bQ_p)\bigr)\to \hH^1_\cont\bigl(G_K, V\otimes_{\bQ_p} (\BB_{\dR,n}^+/\bQ_p)\bigr)$ is then an isomorphism, where $\BB_{\dR,n}^+$ denotes $\BB_{\dR}^+/t^n$. 

Then \cite[Proposition 6.8(ii)]{fontaine-presque} implies that we have a perfect pairing of finite-dimensional $\bQ_p$-vector spaces
\begin{equation*}
\hH^1_\cont\bigl(G_K,V\otimes_{\bQ_p} (\BB_{\dR,n}^+/\bQ_p)\bigr)\otimes_{\bQ_p}\Ext^1_{\cC(G_K)}\bigl(\BB_{\dR,n}^+/\bQ_p, V^{\vee}(1)\bigr)\to \hH^2_\cont\bigl(G_K,\bQ_p(1)\bigr)\simeq\bQ_p
\end{equation*}
where $\Ext^1$ is taken in Fontaine's category of almost $\bC_p$-representations of $G_K$ (see \cite[\textsection 1, 1.2]{fontaine-almost}). Therefore, the result follows from Lemma~\ref{lemma:formula-for-ext1} applied to $W=V^{\vee}(1)$ and a sufficiently large $n$. 
\end{proof}

\begin{lemma}\label{lemma:formula-for-ext1} Let $W\in \Rep_{\QQ_p} G_K$. Then there is an integer $n_0$ such that, for every $n\geq n_0$, there is an isomorphism $\Ext^1_{\cC(G_K)}(\BB_{\dR,n}^+/\bQ_p, W) \simeq D_\cris(W)^{\varphi=1}$.    
\end{lemma}

The proof below involves formulas containing both Tate twists by $\QQ_p(1)$ and twists by the line bundle $\cO(1)$ on the Fargues--Fontaine curve. In order to avoid any confusion, we write $(n)_T$ for Tate twists and $(n)_{FF}$ for twists by $\cO(n)$.

\begin{proof}
To calculate this $\Ext^1$ group, we use the equivalence 
\begin{equation}\label{eqn:derived-almost-representations}
D^b\bigl(\cC(G_K)\bigr)\simeq D^b\bigl(\Coh^{G_K}(\FF_{\bC_p})\bigr)
\end{equation}
of the derived category of $\cC(G_K)$ with the derived category of (continuously) Galois-equivariant coherent sheaves on the Fargues--Fontaine curve provided by \cite[Theorem 6.4]{fontaine-almost}. Recall that for any integer $n$ we have a $G_K$-equivariant line bundle $\cO(n)_{{FF}}$ on $\FF_{\bC_p}$ associated to the $\varphi$-module $(K_0,p^{-n}\cdot\varphi_{K_0})$ over $K_0$. We have Galois-equivariant isomorphisms $\hH^0(\FF_{\bC_p},\cO(n)_{FF})\simeq \BB_{\cris}^{+,\varphi=p^n}$ and $\hH^1(\FF_{\bC_p},\cO(-n)_{FF})\simeq (\BB_{\dR,n}^+/\bQ_p)(-n)_T$ for $n>0$. 

As explained in the proof of \cite[Proposition 6.8]{fontaine-almost} (see also \cite[Theorem 4.23]{Li-duality}), the equivalence (\ref{eqn:derived-almost-representations}) carries the object $\cO(-n)_{FF}\in \Coh^{G_K}(\FF_{\bC_p})$ to $\hH^1(\FF_{\bC_p},\cO(-n)_{FF})[-1]$ for all $n>0$, hence we obtain an isomorphism
\begin{align*}
\Ext^1_{\cC(G_K)}(\BB_{\dR,n}^+/\bQ_p, W) & \simeq \Ext^1_{\cC(G_K)}((\BB_{\dR,n}^+/\bQ_p)(-n)_{T},W(-n)_{T}) \\
& \simeq \Ext^1_{D^b(\Coh^{G_K}(\FF_{\bC_p}))}(\cO(-n)_{FF}[1],W(-n)_{T}\otimes_{\bQ_p}\cO)
\end{align*}
The final $\Ext^1$ is simply the space $\Hom_{\FF_{\bC_p}}(\cO(-n)_{FF},W(-n)_{T}\otimes_{\QQ_p}\cO)^{G_K}$ of Galois-equivariant homomorphisms of vector bundles on $\FF_{\bC_p}$, which equals to the space of Galois-invariant global sections
\[
\hH^0(\FF_{\bC_p},W(-n)_{T}\otimes_{\bQ_p}\cO(n)_{FF})^{G_K}\simeq (\BB_{\cris}^{+,\varphi=p^n}\otimes_{\QQ_p} W(-n)_{T})^{G_K}.
\]
Since $(W\otimes_{\QQ_p} B_\cris^{\varphi=1})^{G_K}$ is finite-dimensional and $B_\cris=B_\cris^+[1/t]$, one sees that the map $\bigl(W\otimes_{\QQ_p} \BB_{\cris}^{+,\varphi=p^n}(-n)_T\bigr)^{G_K}\xrightarrow{\id\otimes t^{-n}} \bigl(W\otimes_{\QQ_p} \BB_{\cris}^{\varphi=1}\bigr)^{G_K}$ is an isomorphism for large enough $n$. Combining the results, we conclude that 
\[
\Ext^1_{\cC(G_K)}(\BB_{\dR,n}^+/\bQ_p, W) \simeq \bigl(W\otimes_{\QQ_p} \BB_{\cris}^{+,\varphi=p^n}(-n)_T\bigr)^{G_K} \simeq \bigl(W\otimes_{\QQ_p} \BB_{\cris}^{\varphi=1}\bigr)^{G_K} \simeq D_\cris(W)^{\varphi=1}
\]
for sufficiently large $n$. 
\end{proof}

\subsection{\texorpdfstring{$p$}{}-adic \'etale cohomology as an object of \texorpdfstring{$\cD^{b}(\Rep_{\bQ_p}^{\dR}G_K)$}{}.} In this subsection we apply the description of $\cD^{b}(\Rep_{\bQ_p}^{\dR}G_K)$ provided by Theorem \ref{thm: de rham derived category pullback} to upgrade $p$-adic \'etale cohomology $\RGamma_{\et}(X_{\bC_p},\bQ_p)$ of any smooth proper rigid-analytic space over $K$ to an object of this category. 

\begin{pr}\label{prop: qp etale cohomology complex is de rham}
Let $X$ be any smooth proper rigid-analytic space over $K$. Then its $p$-adic \'etale cohomology algebra $\RGamma_{\et}(X_{\bC_p},\bQ_p)\in \cD(G_K;\bQ_p)$ naturally lifts to an $E_{\infty}$-algebra in the category $\cD^b(\Rep^{\dR}_{\bQ_p}G_K)$. 
\end{pr}

\begin{rem}\label{rem: wiesia knows it}
We expect that {Proposition \ref{prop: qp etale cohomology complex is de rham} can also be deduced from the results of \cite{Deglise-Niziol}. When $X$ arises as the analytification of an algebraic variety, this result follows from \cite[Lemma 2.22]{Deglise-Niziol}, which produces the requisite object of $\cD^b(\Rep_{\bQ_p}^{\dR}G_K)$ by applying the functor $V_{\pst}$ to an object of the derived category of admissible filtered $(\varphi,N,G_K)$-modules constructed out of Hyodo--Kato cohomology. For a general $X$, the same strategy {should give} a proof, using the construction of Hyodo--Kato cohomology provided in \cite[Theorem 1.1]{Colmez-Niziol-I}.}
\end{rem}

First, we recall how the complex of $\bQ_p$-vector spaces $\RGamma_{\et}(X_{\bC_p},\bQ_p)$ is endowed with a structure of an object of $\cD(G_K;\bQ_p)$. Denote by $f\colon X_{\qproet}\to (\Spa K)_{\qproet}$ the structure map of quasi-pro\'etale sites, as defined in \cite[Definition~14.1]{diamonds}. We have an object $\rR f_*\bQ_p\in \cD((\Spa K)_{\qproet};\bQ_p)$ whose value on $\Spa C$ is the \'etale cohomology complex $\RGamma_{\et}(X_{\bC_p},\bQ_p)$, by the comparison between quasi-pro\'etale and \'etale cohomology for quasi-compact rigid-analytic spaces. Moreover, the complex $\rR f_*\bQ_p$ has a natural structure of an $E_{\infty}$-algebra in the symmetric monoidal category $\cD((\Spa K)_{\qproet};\bQ_p)$ obtained by applying the lax symmetric monoidal functor $\rR f_*$ to the commutative algebra $\bQ_p\in \cD(X_{\qproet},\bQ_p)$.

To obtain an object of $\cD(G_K;\bQ_p)$ from $\rR f_*\bQ_p$, we identify the quasi-pro\'etale topos of $\Spa K$ with the topos of $G_K$--profinite sets $BG_{K,\proet}$ (see \cite[Definition 4.3.1]{proetale}). As before, we denote its underlying category by $G_K\text{--pfsets} \simeq \Pro(G_K\text{--fsets})$.

\begin{construction} There is a natural functor $G_K\text{--pfsets} \to (\Spa K)_\qproet$ which sends a profinite set $S$ with a continuous action of $G_K$ to the quasi-pro\'etale morphism $(\Spa \bC_p \times S)/\underline{G_K} \to \Spa K$, where the action of $\ud{G_K}$ on $\Spa \bC_p \times S$ is diagonal. This functor sends covers to covers and commutes with pullbacks, so \cite[\href{https://stacks.math.columbia.edu/tag/00X6}{Tag 00X6}]{stacks-project} ensures that this defines a morphism of sites
\[
\pi \colon (\Spa K)_\qproet \to BG_{K,\proet}.
\]
\end{construction}

\begin{lemma}\label{lemma:quasiproetale-G-sets} The morphism of sites $\pi \colon (\Spa K)_\qproet \to BG_{K,\proet}$ induces an equivalence of topoi $\pi_* \colon \Shv\bigl((\Spa K)_\qproet; \Set\bigr) \to \Shv\bigl(BG_{K,\proet}; \Set\bigr)$. 
\end{lemma}
\begin{proof}
    We note that $(\Spa K)_{\fet} \simeq (\Spa K)_{\et, \qc, \sep}$. Therefore, \cite[Proposition 11.23 and 11.24]{diamonds} imply that the natural functor 
    \[
    \Pro\bigl((\Spa K)_{\fet}\bigr) \simeq \Pro\bigl((\Spa K)_{\et, \qc, \sep}\bigr) \hookrightarrow (\Spa K)_\qproet
    \]
    is fully faithful and defines a basis for the quasi-pro\'etale topology on $\Spa K$. Therefore it induces an equivalence of topoi $\Shv\bigl(\Pro(\Spa K)_{\fet};\Set\bigr)\to \Shv\bigl((\Spa K)_{\qproet}; \Set\bigr)$.
    
    By Galois theory, the category $(\Spa K)_{\fet}$ is equivalent to the category $G_K\text{--fsets}$, hence the site $\Pro((\Spa K)_{\fet})$ is equivalent to $\Pro(G_K\text{--fsets})\simeq BG_{K,\proet}$. Hence we obtain an equivalence of topoi $\Shv\bigl(BG_{K,\proet};\Set\bigr)\simeq \Shv\bigl((\Spa K)_{\qproet};\Set\bigr)$ which is readily seen to be induced by the morphism $\pi$.   
\end{proof}
Applying the equivalence of topoi from Lemma \ref{lemma:quasiproetale-G-sets}, the $E_{\infty}$-algebra $\rR f_*\bQ_p\in \cD((\Spa K)_{\qproet},\bQ_p)$ is carried to an $E_{\infty}$-algebra in $\cD(G_K;\bQ_p)$ that we denote by $\RGamma_{\et}(X_{\bC_p},\bQ_p)$ as well. We now turn to proving that this object can be presented by a complex of de Rham representations:
\begin{proof}[Proof of Proposition \ref{prop: qp etale cohomology complex is de rham}]
First, the structure of a condensed $G_K$-module on each cohomology group $\hH^i_{\et}(X_{\bC_p},\bQ_p)$ arises from a continuous finite-dimensional $\bQ_p$-representation {(see \cite[Proposition 4.9]{bosco})}, hence $\RGamma_{\et}(X_{\bC_p},\bQ_p)$ lies in the full subcategory $\cD^b(\Rep_{\bQ_p} G_K)\subset \cD(G_K;\bQ_p)$ by Lemma \ref{lem: repgk derived of its heart}.

By Theorem \ref{thm: de rham derived category pullback}, we need to lift $\RGamma_{\et}(X_{\bC_p},\bQ_p)\in \cD^b(\Rep_{\bQ_p}G_K)$ to an $E_{\infty}$-algebra object of the fiber product
\begin{equation*}
\begin{tikzcd}
\cD^b(\Rep_{\bQ_p}G_K)\arrow[rd, "{V\mapsto V\otimes_{\bQ_p}\BB_{\dR}^+}", swap] & & \cD_{\fil}(K)\arrow[ld, "{D\mapsto \Fil^0(D\otimes_K \BB_{\dR})}"] \\
& \cD(G_K;\BB_{\dR}^+)
\end{tikzcd}
\end{equation*}
We will take the corresponding object $D\in\cD_{\fil}(K)$ to be the de Rham cohomology $\RGamma_{\dR}(X/K)$ equipped with its Hodge filtration. It is made into a commutative algebra object of $\cD_{\fil}(K)$ by applying the lax symmetric monoidal functor $\RGamma(X,-)$ to the filtered commutative differential graded algebra $\cO_X\xrightarrow{d}\Omega^1_{X/K}\xrightarrow{d}\ldots$ in sheaves of $K$-vector spaces on $X$.

To lift the pair $\RGamma_{\et}(X_{\bC_p},\bQ_p), \RGamma_{\dR}(X/K)$ to an object of the above fiber product, it remains to construct an equivalence of commutative algebra objects
\[
\RGamma_{\et}(X_{\bC_p},\bQ_p)\otimes_{\bQ_p}\BB_{\dR}^+\simeq \Fil^0(\RGamma_{\dR}(X/K)\otimes_K \BB_{\dR})
\]
in $\cD(G_K;\BB_{\dR}^+)$. It is obtained by passing to $\Fil^0$ in the de Rham comparison isomorphism for rigid-analytic spaces \cite[Theorem 8.4]{scholze-rigid} providing an equivalence of filtered commutative algebra objects in $\cD(G_K;\BB_{\dR})$:
\begin{equation}\label{de rham comparison formula}
\RGamma_{\et}(X_{\bC_p},\bQ_p)\otimes_{\bQ_p}\BB_{\dR}\simeq \RGamma_{\dR}(X/K)\otimes_K \BB_{\dR}.
\end{equation}
In loc.~cit.~this equivalence is stated on the level of individual cohomology groups, but it is constructed by proving that the natural maps \begin{multline*}\RGamma_{\et}(X_{\bC_p},\bQ_p)\otimes_{\QQ_p} \BB_{\dR}\to \RGamma_{\proet}(X_{\bC_p},\bB_{\dR}) \\ \RGamma_{\proet}(X_{\bC_p},\bB_{\dR})\to \RGamma_{\proet}(X_{\bC_p},\cO\bB_{\dR}\xrightarrow{\nabla}\ldots) \\ \RGamma_{\dR}(X/K)\otimes_K \BB_{\dR}\to \RGamma_{\proet}(X_{\bC_p},\cO\bB_{\dR}\xrightarrow{\nabla}\ldots)\end{multline*} of filtered commutative algebras are equivalences, which induces the equivalence (\ref{de rham comparison formula}).
\end{proof}

\section{Formality from the weight-monodromy conjecture}

\subsection{Formulation of the main theorem}

Let $K$ be a finite extension of $\QQ_p$, and let $\ell$ be a prime number (possibly equal to $p$). In this subsection, we formulate our main theorem, which will be proven in Subsection~\ref{subsection:proof-formality}.

\begin{thm}\label{thm:main-formality}
Let $X$ be a smooth proper rigid-analytic space over $K$. Assume that the weight-monodromy conjecture holds for $\hH^i_\et(X_{\bC_p}, \QQ_\ell)$ for all $i\geq 0$. Then the $E_{\infty}$-algebra $\RGamma_{\et}(X_{\bC_p},\bQ_{\ell})$ is formal.
\end{thm}

Definition~\ref{defn:weight-monodromy} explains what it means for ``the weight-monodromy conjecture to hold for $\hH^i_\et(X_{\bC_p}, \QQ_\ell)$'' and Theorem~\ref{thm:main-formality} is proven in Subsection~\ref{subsection:proof-formality}. Together with the following lemma relating formality of $\bQ_p$-\'etale cohomology with that of de Rham cohomology, Theorem~\ref{thm:main-formality} gives the main Theorem~\ref{thm: wm implies formality intro} from the introduction.

\begin{lm}\label{dr vs padic etale}
Let $X$ be a smooth proper rigid-analytic space over $K$.  Then the $E_\infty$-algebra $\RGamma_{\et}(X_{\bC_p},\bQ_p)$ is formal if and only if the de Rham cohomology $E_{\infty}$-algebra $\RGamma_{\dR}(X/K)$ is formal. 
\end{lm}
\begin{proof}
As discussed in the proof of Proposition \ref{prop: qp etale cohomology complex is de rham}, we have an equivalence of $E_{\infty}$-algebras over the field $\BB_{\dR}$
\begin{equation}\label{eqn: bdr comparison formula}
\RGamma_{\et}(X_{\bC_p},\bQ_p)\otimes_{\bQ_p}\BB_{\dR}\simeq \RGamma_{\dR}(X/K)\otimes_K \BB_{\dR}.
\end{equation}
Formality of an $E_{\infty}$-algebra can be checked after an arbitrary extension of characteristic zero fields \cite[Corollary 6.9]{Obstructions}, hence the identification (\ref{eqn: bdr comparison formula}) proves the lemma.
\end{proof}

\subsection{Weil--Deligne representations}

In this subsection, we fix a finite extension $\QQ_p\subset K$ with Galois group $G_K:=\Aut(\oK/K)$, a prime number $\ell$, and a field $E$ of characteristic $0$ that will serve as the field of coefficients. One should feel free to assume that $E=\QQ_\ell$ when $\ell\neq p$ and that $E = \QQ_p^{\nr}$ when $\ell=p$. We denote by $W_K$ the Weil group of $K$ and by $q$ the size of the residue field $k=\bF_q$ of $K$. The main goal of this subsection is to recall the notion of Weil--Deligne representations and prove its basic properties. 

An \emph{$E$-linear Weil--Deligne representation} is a finite-dimensional $E$-vector space equipped with an action of the Weil group $W_K$, on which the inertia acts through a finite quotient, and an $E$-linear map $N \colon V \to V(-1)$ of $W_K$-representations. We denote by $\Rep_E WD_K$ the category of $E$-linear Weil--Deligne representations. 

The category $\Rep_E WD_K$ has a natural symmetric monoidal structure, where $V_1\otimes_E V_2$ is defined to be the usual tensor product of $W_K$-representations and the monodromy operator $N_{V_1\otimes_E V_2}$ is defined as $N_{V_1}\otimes \id_{V_2} + \id_{V_1}\otimes N_{V_2}$. This makes $\Rep_E WD_K$ into a symmetric monoidal abelian category. We denote by $\cD(\Rep_E WD_K)$ the associated symmetric monoidal derived $\infty$-category (see \cite[Proposition A.5]{Nikolaus-Scholze}).

We now discuss some examples of Weil--Deligne representations attached to (some) Galois representations. 

\begin{examples}\label{ex: wd from galois reps}
    \begin{enumerate}
    \item Let $\ell\neq p$. Then there is a fully faithful exact symmetric monoidal functor 
    \[
    \DD \colon \Rep_{\QQ_\ell} G_K \to \Rep_{\QQ_\ell} WD_K.
    \]
    We refer the reader to \cite[\textsection 2.2]{Fon-ell} and \cite[Example 2.3]{betts-litt} for the construction of this functor. Choosing a compatible system of $\ell$-power roots of unity in $\oK$ and a Frobenius element, we may identify the underlying vector space of $\DD(V)$ with $V$ for every $V\in \Rep_{\bQ_\ell}G_K$, compatibly with tensor product. The exactness of this functor and \cite[Proposition A.5]{Nikolaus-Scholze} allow us to derive it to a $t$-exact symmetric monoidal functor $\DD\colon \cD(\Rep_{\QQ_\ell} G_K) \to \cD(\Rep_{\QQ_\ell} WD_K)$.
    \item Let $\ell = p$. Then there is an exact symmetric monoidal functor 
    \[
    \DD \colon \Rep^\dR_{\QQ_p} G_K \to \Rep_{\QQ_p^\nr} WD_K,
    \]
    sending a de Rham representation $V$ to the Weil--Deligne representation constructed out of the $(\varphi,N,G_K)$-module $D_{\pst}(V)$. Here, we use that every de Rham representation is potentially semi-stable \cite[Th\'eor\`eme 0.7]{Berger}, so that $\DD$ is faithful and monoidal. We refer to \cite[\textsection 4.2]{fontaine-semistables} and \cite[Example 2.4]{betts-litt} for the details of the definition of this functor. For any $V\in \Rep_{\bQ_p}^{\dR}G_K$ there is a natural isomorphism \[V\otimes_{\bQ_p}B_{\stab,\obQ_p}\simeq \DD(V)\otimes_{\bQ^{\nr}_p}B_{\stab,\obQ_p}\] of $B_{\stab,\obQ_p}\coloneqq B_{\stab}\otimes_{\bQ_p^{\nr}}\obQ_p$-modules compatible with tensor products. Similarly to the $\ell$-adic situation, we can derive to a functor $\DD\colon \cD(\Rep^\dR_{\QQ_p} G_K) \to \cD(\Rep_{\QQ_p^\nr} WD_K)$.
    \end{enumerate}
\end{examples}

Our next goal is to define what it means for cohomology of smooth rigid-analytic spaces to satisfy the weight-monodromy conjecture. For this, we need the following definition. 

\begin{defn}\label{defn:monodromy-pure} An $E$-linear Weil--Deligne representation $V$ is {\it monodromy-pure} of weight $i$ if, for any integer $j$, a (fixed) Frobenius lift $\varphi\in W_K$ acts on $\gr_j^M V$ of the monodromy filtration with eigenvalues that are $q$-Weil numbers of weight $i+j$.
\end{defn}

We refer to \cite[Proposition 1.6.1]{weil2} for the definition of the monodromy filtration. The definition of a monodromy-pure representation is independent of the choice of a Frobenius lift because, for any two Frobenius lifts $\varphi, \varphi'\in W_K$, there is an integer $n\geq 1$ such that the actions of $\varphi^n$ and $\varphi'^n$ on $V$ (and, therefore, on each $\gr^M_i V$) coincide. Furthermore, \cite[Proposition 1.6.9(i)]{weil2} implies that a tensor product of monodromy-pure Weil--Deligne representations of weights $i$ and $j$ is monodromy-pure of weight $i+j$. 

\begin{defn}\label{defn:weight-monodromy} Let $X$ be a smooth proper rigid-analytic space over $K$. We say that \emph{the weight-monodromy conjecture holds for $\hH^i_\et(X_{\bC_p}, \QQ_\ell)$} if the associated Weil--Deligne representation $\DD\bigl(\hH^i_\et(X_{\bC_p}, \QQ_\ell)\bigr)$ is monodromy-pure of weight $i$.
\end{defn}
\begin{rem}
According to the $\ell$-independence conjecture \cite[Conjecture 3.9]{scholze-cdm}, if the weight-monodromy conjecture for a given $X$ and $i$ holds for one prime $\ell$, then it should hold for every other prime as well.
\end{rem}

Now we consider the functor $\rR\Gamma(WD_K, -) \coloneqq \rR \Hom_{WD_K}(E, -) \colon \cD(\Rep_E WD_K) \to \cD(\Vect_{E})$. We also set $\hH^i(WD_K, M) \coloneqq \hH^i\bigl(\rR\Gamma(WD_K, M)\bigr)$ for any $M\in \Rep_E WD_K$. Our last main goal of this subsection is to present a complex computing $\rR\Gamma(WD_K, M)$ for any $M\in \Rep_E WD_K$.

The following lemma will be used to identify the Ext functors in the category $\Rep_E WD_K$ and its mixed version:

\begin{lemma}\label{lemma:killing-homs-into-E(1)} Let $V, U\in \Rep_E WD_K$ and let $S=E\cdot e_1 \oplus E\cdot e_2$ be a two-dimensional Weil--Deligne representation such that $N_S=0$, the inertia $I_K$ acts trivially on $S$, and any Frobenius lift $\varphi$ acts via the unique $E$-linear transformation such that $\varphi(e_1)=e_1$ and $\varphi(e_2)=e_1+e_2$. Let $V \hookrightarrow V\otimes_E \Sym^r S$ be a homomorphism of Weil--Deligne representations defined via the formula $v\mapsto v\otimes e_1^r$. Then the induced morphism
\[
\Hom_{WD_K}\bigl(V\otimes_E \Sym^r S, U\bigr) \to \Hom_{WD_K}\bigl(V, U\bigr)
\]
is zero for any $r\geq \dim_E V \cdot \dim_E U$.
\end{lemma}
\begin{proof}
Completely analogous to the proof of Lemma \ref{lemma:killing-homs-into-E(1)-galois}.
\end{proof}

Now we show that $\RHom$ in the derived category of finite-dimensional representations of $WD_K$ is computed by a Koszul-type complex:

\begin{lemma}\label{lemma:Weil-Deligne-cohomology} Let $V\in \Rep_E WD_K$ and let $\varphi\in W_K$ be a lift of Frobenius. Then the complex
\[
C^{\bullet}(V)\coloneqq [V^{I_K} \xrightarrow{(\varphi_{V} - \id, N)} V^{I_K}\oplus V(-1)^{I_K} \xrightarrow{(N, \id- \varphi_{V(-1)}) } V(-1)^{I_K}]
\]
is quasi-isomorphic to $\rR\Gamma(WD_K, V)$. 
\end{lemma}
Note that, after identifying $V^{I_K}\simeq V(-1)^{I_K}$ by trivializing $E(-1)$, the Frobenius operator $\varphi_{V(-1)}$ becomes equal to $q\cdot \varphi_V$. 
\begin{proof}
    First, we note that the functor of inertia invariants $(-)^{I_K}$ is exact on $\Rep_E WD_K$ since the action of $I_K$ factors through a finite quotient and $\mathrm{char}~E=0$. This implies that, for any short exact sequence $0 \to V' \to V \to V'' \to 0$, the associated sequence of complexes $0 \to C^\bullet(V') \to C^\bullet(V) \to C^\bullet(V'') \to 0$ is exact. Thus, the sequence $\{V \mapsto \hH^i\bigl(C^\bullet(V)\bigr)\}_{i\in \ZZ_{\geq 0}}$ is naturally a $\delta$-functor. Lemma~\ref{lemma:exts-are-universal} implies that it suffices to show that there is a functorial isomorphism $\Hom_{WD_K}(E, V) \simeq \hH^0\bigl(C^\bullet(V)\bigr)$ for any $V$ and that $\{V \mapsto \hH^i\bigl(C^\bullet(V)\bigr)\}_{i\in \ZZ_{\geq 0}}$ is a universal $\delta$-functor.

    The first claim is clear since $\hH^0\bigl(C^\bullet(V)\bigr) \simeq V^{N=0, W_K} \simeq \Hom_{WD_K}(E, V)$ naturally in $V$. Therefore, it suffices to show that $\{V \mapsto \hH^i\bigl(C^\bullet(V)\bigr)\}_{i\in \ZZ_{\geq 0}}$ is a universal $\delta$-functor. Now \cite[Proposition~4.2]{Buchsbaum} implies that it suffices to show that, for any $V\in \Rep_E WD_K$ and a class $u\in \hH^i\bigl(C^\bullet(V)\bigr)$ for $i=1$ or $2$, there is an injection $V\hookrightarrow V'$ such that the image of $u$ in $\hH^i\bigl(C^\bullet(V')\bigr)$ is zero. 

    For $i=1$, we suppose that $u$ is represented by an element $(x, y)\in V^{I_K}\oplus V^{I_K}(-1)$. Then we take $V'\coloneqq V\oplus E\cdot e$ and endow it with the unique structure of a Weil--Deligne representation such that $V\subset V'$ is a sub-representation, and $e$ is invariant under all of $I_K$ and we have $N_{V'}(e)=y,\varphi_{V'}(e)=x\oplus e$. This Weil--Deligne representation is well-defined by the assumption that $(x,y)$ is a cocycle in $C^1(V)$. One readily checks that the map $H^1(C^{\bullet}(V))\to H^1(C^{\bullet}(V'))$ induced by the natural morphism $V\hookrightarrow V'$ kills the class of $u$.

    \comment{promote it to a Wel--Deligne representation by defining the monodromy operator via the formula $N_{V'}(v\oplus c\cdot e)=N(v)+cy$ and the Weil action via the inductive formula 
    \[
    \varphi^n\sigma(v\oplus c\cdot e)= \begin{cases} 
  \varphi^n\sigma(v) + c\cdot \varphi^{n-1}(e+x) & \text{if } n > 0 \\
  \sigma(v) + c\cdot e & \text{if } n= 0 \\
  \varphi^n\sigma(v) + c\cdot \varphi^{n+1}(e-\varphi^{-1}(x)) & \text{if } n < 0 
\end{cases}
    \]
    for any $\varphi^n\sigma\in W_K$ with $\sigma\in I_K$. Then one readily checks that $V'$ is a well-defined Weil--Deligne representation and the natural morphism $V \xhookrightarrow{v\mapsto v\oplus 0} V'$ kills the class of $u$. }

    For $i=2$, we note that there is an evident isomorphism of complexes $C^\bullet(V)^{\vee}[-2]\simeq C^\bullet(V^{\vee}(1))$. By passing to cohomology, we conclude that $\hH^2\bigl(C^\bullet(V)\bigr)^\vee \simeq \hH^0\bigl(C^\bullet(V^\vee(1))\bigr) \simeq \Hom_{WD_K}(E, V^\vee(1))\simeq \Hom_{WD_K}(V, E(1))$. Therefore, it suffices to show that, for any $V\in \Rep_E WD_K$, we can find an injection $V \hookrightarrow V'$ such that the natural morphism $\Hom_{WD_K}\bigl(V', E(1)\bigr) \to \Hom_{WD_K}\bigl(V, E(1)\bigr)$ is zero. This is achieved by taking $V'\simeq V\otimes_E \Sym^{\dim_E V} S$ from Lemma~\ref{lemma:killing-homs-into-E(1)}. 
\end{proof}

\begin{cor}\label{corollary:Weil-Deligne-Tate-duality} Let $V\in \Rep_E WD_K$. Then $\rR\Gamma(WD_K, V)$ is concentrated in cohomological degrees $[0,2]$ and each of its cohomology groups is finite-dimensional over $E$. Also, there is a natural equivalence $\rR\Gamma(WD_K, V)^{\vee} \simeq \rR\Gamma(WD_K, V^\vee(1)[2])$.
\end{cor}
\begin{proof}
    The first claim follows directly from the isomorphism $\rR\Gamma(WD_K, V)\simeq C^{\bullet}(V)$ from Lemma~\ref{lemma:Weil-Deligne-cohomology}. The second claim follows immediately from the isomorphism of complexes $C^\bullet(V)^{\vee}\simeq C^\bullet(V^{\vee}(1))[2]$. 
\end{proof}

\subsection{Mixed Weil--Deligne representations}

In this subsection, we fix a finite extension $\QQ_p\subset K$ with Galois group $G_K$, a prime number $\ell$, and a field $E$ of characteristic $0$. The main goal of this subsection is to recall the notion of mixed and monodromy-pure mixed Weil--Deligne representations and prove that two versions of the derived $\infty$-categories of mixed Weil--Deligne representations coincide. 

\begin{defn} A \emph{mixed Weil--Deligne representation} is a pair $(V, W_\bullet V)$ consisting of an object $V\in \Rep_E WD_K$ and a complete exhaustive increasing $\bZ$-indexed filtration $W_i V\subset V$ in $\Rep_E WD_K$ such that each $\gr^W_i V \coloneqq W_iV/W_{i-1}V$ is monodromy-pure of weight $i$. 

The \emph{weights} of a mixed Weil--Deligne representation $(V, W_\bullet V)$ are the integers $i\in \ZZ$ such that $\gr^W_i V\neq 0$. We denote the \emph{category of mixed Weil--Deligne representations} by $\Rep^\mix_E WD_K$. 
\end{defn}

Since the tensor product of monodromy-pure Weil--Deligne representations of weights $i$ and $j$ is monodromy-pure of weight $i+j$, one readily sees that the tensor product of two mixed Weil--Deligne representations has a natural structure of a mixed Weil--Deligne representation. So $\Rep_E^\mix WD_K$ is a symmetric monoidal category. Remarkably, the category $\Rep_E^\mix WD_K$ happens to be abelian \cite[Proposition 20]{vologodsky} (see also \cite[Theorem 2.10]{betts-litt}) because any morphism of mixed representations is strictly compatible with weight filtrations. Moreover, the forgetful functor $\Rep^\mix_E WD_K \to \Rep_E WD_K$ is symmetric monoidal and exact. 

In particular, the tensor product on $\Rep^\mix_E WD_K$ is exact in each variable. Thus, \cite[Proposition A.5]{Nikolaus-Scholze} ensures that $\cD(\Rep^\mix_E WD_K)$ is a symmetric monoidal $\infty$-category and the forgetful functor $\cD(\Rep^\mix_E WD_K) \to \cD(\Rep_E WD_K)$ is symmetric monoidal and $t$-exact. 

Since the functors of inner Hom are exact on both $\Rep_E^\mix WD_K$ and $\Rep_E WD_K$, we can extend them to $t$-exact functors
\[
\ud{\Hom}(-,-)\colon \cD(\Rep^\mix_E WD_K)^{\op} \times \cD(\Rep^\mix_E WD_K) \to \cD(\Rep^\mix_E WD_K),
\]
\[
\ud{\Hom}(-,-)\colon \cD(\Rep_E WD_K)^{\op} \times \cD(\Rep_E WD_K) \to \cD(\Rep_E WD_K),
\]
which are compatible with respect to the forgetful functors. 

\begin{defn}\label{defn:mixed-monodromy-pure} A mixed Weil--Deligne representation $(V, W_\bullet V)\in \Rep^\mix_E WD_K$ is \emph{monodromy-pure of weight $i$} if $W_{i-1}V=0$ and $W_i V= V$. 
\end{defn}

This definition clearly implies that the underlying Weil--Deligne representation is monodromy-pure of weight $i$. However, the property introduced in this definition is stronger. For instance, one can take $V$ to be a non-trivial extension $0 \to \QQ_\ell(1) \to V \to \QQ_\ell \to 0$ equipped with the filtration $W_{-3}V=0$, $W_{-2}V = W_{-1}V =\QQ_\ell(1)$, and $W_0V = V$. Then $V$ is monodromy-pure of weight $-1$, but $(V, W_\bullet V)\in \Rep_{\bQ_{\ell}}^{\mix}WD_K$ is not monodromy-pure of any weight. 

\begin{defn}\label{defn:derived-pure-reps} We define $\cD_\pure(\Rep_E^\mix WD_K) \subset \cD(\Rep_E^\mix WD_K)$ to be the full $\infty$-subcategory consisting of objects $M$ such that $\hH^i(M)$ is a monodromy-pure \underline{mixed} Weil--Deligne representation of weight $i$ for any $i\in \ZZ$ (see Definition~\ref{defn:mixed-monodromy-pure}). 

Likewise, we define $\cD_\pure(\Rep_E WD_K) \subset \cD(\Rep_E WD_K)$ to be the full $\infty$-subcategory consisting of objects $M$ such that $\hH^i(M)$ is a monodromy-pure Weil--Deligne representation of weight $i$ for any $i\in \ZZ$ (see Definition~\ref{defn:monodromy-pure}).
\end{defn}

Since the tensor products on $\cD(\Rep_E^\mix WD_K)$ and  $\cD(\Rep_E WD_K)$ are $t$-exact in each variable and the tensor product of (mixed) monodromy-pure Weil--Deligne representations is (mixed) monodromy-pure of the correct weight, we conclude that $ \cD_\pure(\Rep_E^\mix WD_K)$ and $\cD_\pure(\Rep_E WD_K)$ are symmetric monoidal $\infty$-subcategories of $\cD(\Rep_E^\mix WD_K)$ and $\cD(\Rep_E WD_K)$, respectively. 

Therefore, the symmetric monoidal forgetful functor derives to a symmetric monoidal functor $F\colon \cD_\pure(\Rep_E^\mix WD_K) \to \cD_\pure(\Rep_E WD_K)$. We wish to show that this functor is an equivalence on bounded objects.
\begin{thm}\label{thm: dbpure is dbpure of mixed} The forgetful functor $F\colon \cD^b_{\pure}(\Rep_E^\mix WD_K)\to \cD^b_{\pure}(\Rep_E WD_K)$ is a symmetric monoidal equivalence.
\end{thm}

We start with the following lemma. 

\begin{lemma}\label{lemma:lift-to-mix} Let $(V', W_\bullet V')\in \Rep^\mix_E WD_K$ be a mixed Weil--Deligne representation of non-positive weights and let $0 \to V'\to V \to V'' \to 0$ be an extension in $\Rep_E WD_K$ such that $V''$ is monodromy-pure of weight $0$. Set $W_i V = W_iV'$ if $i<0$ and $W_iV = V$ if $i\geq 0$. Then $(V, W_\bullet V)$ is a mixed Weil--Deligne representation of non-positive weights. 
\end{lemma}
\begin{proof}
    The only property we need to check is that $\gr^W_0 V$ is monodromy-pure of weight $0$. By construction, it fits into the short exact sequence $0 \to \gr^W_0 V' \to \gr^W_0 V \to V'' \to 0$, where both $\gr^W_0 V'$ and $V''$ are monodromy-pure of weight $0$. This follows from (the proof of) \cite[Proposition 2.11]{betts-litt}.
\end{proof}

Now we begin the proof of the fact that the two versions of $\cD_\pure$ are equivalent. As the first step, we need a formula for the Ext groups in $\cD(\Rep_E^\mix WD_K)$. For this, we recall that \cite[Theorem 2.10]{betts-litt} ensures that the functor $W_0 \colon \Rep^\mix_E WD_K \to \Rep_E WD_K$ is exact, therefore, it extends to a $t$-exact functor 
\[
W_0 \colon \cD(\Rep^\mix_E WD_K) \to \cD(\Rep_E WD_K). 
\]
We show that Exts in the derived category $\cD^b(\Rep_E^{\mix}WD_K)$ coincide with Exts in the filtered derived category of Weil--Deligne representations:
\begin{lm}\label{lemma:exts in wdmix} Let $V_1, V_2\in \cD^b(\Rep^\mix_E WD_K)$. Then there is a natural isomorphism $\RHom_{WD_K^{\mix}}(V_1, V_2)\simeq \RGamma\bigl(WD_K, W_0\underline{\Hom}(V_1, V_2)\bigr)$.
\end{lm}
\begin{proof}
    First, we note that $\RHom_{WD_K^{\mix}}(V_1, V_2) \simeq \RHom_{WD_K^\mix}\bigl(E, \ud{\Hom}(V_1, V_2)\bigr)$. Since $W_0 E =E$, the functor $W_0 \colon \cD(\Rep^\mix_E WD_K) \to \cD(\Rep_E WD_K)$ induces a natural morphism
    \[
    \RHom_{WD_K^\mix}\bigl(E, \ud{\Hom}(V_1, V_2)\bigr) \to \RHom_{WD_K}\bigl(E,W_0\ud{\Hom}(V_1, V_2)\bigr)
    \]
    which can be rewritten as a morphism $\RHom_{WD_K^{\mix}}(V_1, V_2) \to \RGamma\bigl(WD_K, W_0\underline{\Hom}(V_1, V_2)\bigr)$. Therefore, we reduce the question to showing that the natural morphism
    \begin{equation}\label{eqn:compute-ext-mix}
    \RHom_{WD_K^\mix}(E, V) \to \RHom_{WD_K}(E,W_0 V)
    \end{equation}
    is an isomorphism for any $V\in \cD^b(\Rep^\mix_E WD_K)$. Since $W_0$ is $t$-exact, we reduce the question to showing that (\ref{eqn:compute-ext-mix}) is an isomorphism for $V\in \Rep^\mix_E WD_K$. In this case, Lemma~\ref{lemma:exts-are-universal} reduces the question to showing that the natural map $\Hom_{WD_K^\mix}(E, V) \to \Hom_{WD_K}(E, W_0V)$ is an isomorphism for any $V\in \Rep^\mix_E WD_K$ and that $\bigl\{\Ext^i_{WD_K}\bigl(E, W_0(-)\bigr)\bigr\}_{i\in \ZZ_{\geq 0}}= \bigl\{\hH^i\bigl(WD_K, W_0(-)\bigr)\bigr\}_{i\in \ZZ_{\geq 0}}$ is a universal $\delta$-functor on $\Rep_E^\mix WD_K$. 

    The claim that the map $\Hom_{WD_K^\mix}(E, V) \to \Hom_{WD_K}(E, W_0V)$ is an isomorphism is clear. So we only need to show the latter claim. For this, we first note that $\bigl\{\hH^i\bigl(WD_K, W_0(-)\bigr)\bigr\}_{i\in \ZZ_{\geq 0}}$ is indeed a $\delta$-functor since $W_0 \colon \Rep^\mix_E WD_K \to \Rep_E WD_K$ is an exact functor.  
    
    Now \cite[Proposition~4.2]{Buchsbaum} ensures that it suffices to show that, for any $V\in \Rep^\mix_E WD_K$, $i\in \ZZ_{>0}$, and $u\in \hH^i\bigl(WD_K, W_0V\bigr)$, there is an injection $V\hookrightarrow V'$ in $\Rep_E^\mix WD_K$ such that the image of $u$ in $\hH^i\bigl(WD_K, W_0V'\bigr)$ vanishes. Furthermore, by Corollary~\ref{corollary:Weil-Deligne-Tate-duality} we only need to treat the cases of $i=1$ and $i=2$.

    \emph{Case of $i=1$.} Fix a class $u\in \hH^1\bigl(WD_K, W_0V\bigr) = \Ext^1_{WD_K}\bigl(E, W_0V\bigr)$. By \cite[\href{https://stacks.math.columbia.edu/tag/06XU}{Tag 06XU}]{stacks-project}, it is represented by an extension $0\to W_0 V \to U \to E \to 0$ in $\Rep_E WD_K$. Lemma~\ref{lemma:lift-to-mix} ensures that $U$ can be naturally promoted to an object $(U, W_\bullet U)\in\Rep^\mix_E WD_K$ such that $W_0U=U$ and such that $0 \to W_0V \to U \to E \to 0$ is exact in $\Rep_E^\mix WD_K$. The image of $u$ in $\hH^1(WD_K, U)=\Ext^1_{WD_K}(E, U)$ is $0$, because the extension $U$ becomes split upon being pushed out to $U$ itself. 
    
    Now we define $V'\in \Rep^\mix_E WD_K$ as the pushout $V\oplus_{W_0V} U\coloneqq \coker(W_0V \to V\oplus U)$ in $\Rep^\mix_E WD_K$. The natural map $V \to V'$ is injective since so is $W_0(V) \to U$. Furthermore, the exactness of $W_0$ implies that $W_0V'\simeq U$. Therefore, the image of $u$ in $\hH^1\bigl(WD_K, W_0V'\bigr)=\hH^1(WD_K, U)$ is zero as shown in the previous paragraph. This finishes the proof when $i=1$. 

    \emph{Case of $i=2$.} We need to show that, for any $V\in \Rep_E^\mix WD_K$, there is an injection $V \hookrightarrow V'$ in $\Rep_E^\mix WD_K$ such that $\hH^2\bigl(WD_K, W_0V\bigr) \to \hH^2\bigl(WD_K, W_0V'\bigr)$ is zero. By Tate duality for Weil--Deligne cohomology (see Corollary~\ref{corollary:Weil-Deligne-Tate-duality}), this is equivalent to asking that 
    \begin{align*}
    \Hom_{WD_K}\bigl(E, (W_0V')^{\vee}(1)\bigr)  & \simeq \Hom_{WD_K}\bigl(W_0V' , E(1)\bigr)  \\
    &\longrightarrow \Hom_{WD_K}\bigl(W_0 V , E(1)\bigr) \simeq \Hom_{WD_K}\bigl(E, (W_0 V)^{\vee}(1)\bigr)
    \end{align*}
    is zero. Let $S$ be the two-dimensional $E$-linear Weil--Deligne representation from Lemma~\ref{lemma:killing-homs-into-E(1)}. Note that it is monodromy-pure of weight $0$, so $(S, W_\bullet S)$ lies in $\Rep^\mix_E WD_K$ with the filtration defined by $W_iS=0$ if $i<0$ and $W_iS=S$ if $i\geq 0$. For any $r\geq 0$, the object $W_0V \otimes_E \Sym^r S \in \Rep_E^\mix WD_K$ has only non-positive weights. Now set $d\coloneqq \dim_E W_0 V$ and $U\coloneqq W_0 V\otimes_E \Sym^{d} S$. It comes with the natural inclusion of Weil--Deligne representations $W_0V \hookrightarrow U$ and Lemma~\ref{lemma:killing-homs-into-E(1)} implies that the natural morphism
    \[
    \Hom_{WD_K}\bigl(U, E(1)\bigr) \to \Hom_{WD_K}\bigl(W_0V, E(1)\bigr)
    \]
    is zero. 

    As in the case $i=1$, we define $V'\in \Rep^\mix_E WD_K$ as the pushout $V\oplus_{W_0V} U\coloneqq \coker(W_0(V) \to V\oplus U)$ in $\Rep^\mix_E WD_K$. The natural map $V \to V'$ is injective and, after applying $W_0$, it induces the map $W_0V\to U$. Hence it induces the zero map on $\hH^2(WD_K,-)$, as we just established. This finishes the proof. 
\end{proof}

Before demonstrating the equivalence of the two versions of pure complexes, we only need one additional preliminary lemma.

\begin{lemma}\label{lemma:equality-of-homs-weights} Let $V_1, V_2\in \cD^b_{\pure}(\Rep_E^\mix WD_K)$ such that $V_1\in \cD^{\geq n}_{\pure}(\Rep_E^\mix WD_K)$ is concentrated in cohomological degrees $\geq n$ and $V_2\in \cD^{\leq n}_{\pure}(\Rep_E^\mix WD_K)$  is concentrated in degrees $\leq n$. Then the natural morphism $W_0\underline{\Hom}(V_1, V_2)\to \underline{\Hom}(V_1, V_2)$ is an isomorphism.
\end{lemma}
\begin{proof}
First, we have $\underline{\Hom}(V_1,V_2)\simeq V_1^{\vee}\otimes_E V_2$. Furthermore, the dual object $V^\vee_1$ belongs to $\cD^{\leq -n}_{\pure}(\Rep^{\mix}_E WD_K)$, hence the tensor product $V_1^{\vee}\otimes_E V_2$ lies in $\cD^{\leq 0}_{\pure}(\Rep^{\mix}_E WD_K)$. Thus, it suffices to show that the natural morphism $W_0 V \to V$ is an isomorphism for any $V \in \cD_{\pure}^{\leq 0}(\Rep_E^\mix WD_K)$. This can be checked on individual cohomology groups, where the claim is clear.
\end{proof}

Recall that, for an $\infty$-category $\cC$, we denote by $\Map_\cC(X, Y)$ the mapping anima between objects $X$ and $Y$. If $\cC=\cD(\cA)$ for an abelian category $\cA$, this mapping anima can be identified with $\tau^{\leq 0}\RHom_{\cA}(X, Y)$.

\begin{proof}[Proof of Theorem \ref{thm: dbpure is dbpure of mixed}]
    The forgetful functor is symmetric monoidal, so we only need to show that it is an equivalence of underlying $\infty$-categories. 

    First, we show that $F$ is fully faithful. For this, we need to show that the natural morphism $\Map_{\cD(\Rep_E^\mix WD_K)}(V_1, V_2) \to \Map_{\cD(\Rep_E WD_K)}(V_1, V_2)$ is an equivalence for any $V_1, V_2\in \cD^b_{\pure}(\Rep_E^\mix WD_K)$. In other words, we wish to show that the natural morphism
    \[
    \tau^{\leq 0}\RHom_{WD_K^\mix}(V_1, V_2) \to \tau^{\leq 0}\RHom_{WD_K}(V_1, V_2)
    \]
    is an isomorphism for any $V_1, V_2\in \cD^b_{\pure}(\Rep_E^\mix WD_K)$. Since $\RHom_{WD_K^\mix}(V_1, V_2) \simeq \RHom_{WD_K^\mix}(E, V_1^\vee\otimes_E V_2)$ and the same formula holds in $\cD(\Rep_E WD_K)$, we can assume that $V_1=E$. In this case, we denote $V_2$ simply by $V$.

   In this situation, we consider the following commutative diagram
    \[
    \begin{tikzcd}
        \RHom_{WD_K^\mix}(E, \tau^{\leq 0} V) \arrow{d}{\alpha} \arrow{r} & \RHom_{WD_K^\mix}(E,  V) \arrow{d}{\beta} \arrow{r} & \RHom_{WD_K^\mix}(E, \tau^{>0} V) \arrow{d}{\gamma} \\
        \RHom_{WD_K}(E, \tau^{\leq 0} V) \arrow{r} & \RHom_{WD_K}(E, V) \arrow{r} & \RHom_{WD_K}(E, \tau^{>0} V),
    \end{tikzcd}
    \]
    where the rows are fiber sequences. We wish to show that $\tau^{\leq 0}\beta$ is an isomorphism.  Since both $\RHom_{WD_K^\mix}(E, \tau^{>0} V)$ and $\RHom_{WD_K}(E, \tau^{>0} V)$ lie in $\cD^{\geq 1}(\ZZ)$, it suffices to show that $\alpha$ is an isomorphism. This follows directly from Lemma~\ref{lemma:exts in wdmix} and Lemma~\ref{lemma:equality-of-homs-weights}. 

    Now we show that $F$ is essentially surjective. First, we note that, by the very definition, every object of the form $V[-d]$ for a monodromy-pure $V$ of weight $d$ lies in the essential image of $F$. Now we pick $V\in \cD^{[a,b]}_{\pure}(\Rep_E WD_K)$ and wish to show that it lies in the essential image. We argue by induction on $b-a\geq 0$. If $b-a=0$, this was proved above. So we assume that $n\coloneqq b-a>0$ and the claim is known for any $b-a<n$. Consider the exact triangle $\hH^a(V)[-a] \to V \to \tau^{>a}V$. By induction, $\hH^a(V)[-a]$ and $\tau^{>a}V$ lie in the essential image of $F$. Say $\hH^a(V)[-a]  \simeq F(U_1)$ and $\tau^{>a}V \simeq F(U_2)$. Since $F$ is $t$-exact, we conclude that $U_1\in \cD^{[a, a]}_{\pure}(\Rep_E^\mix WD_K)$ and $U_2\in \cD^{[a+1, b]}_{\pure}(\Rep_E^\mix WD_K)$. Now it suffices to show that the natural map
    \[
    \Ext^1_{WD_K^\mix}\bigl(U_2, U_1\bigr) \to \Ext^1_{WD_K}\bigl(F(U_2), F(U_1)\bigr)
    \]
    is an isomorphism. This follows immediately from Lemma~\ref{lemma:exts in wdmix} and Lemma~\ref{lemma:equality-of-homs-weights}.
\end{proof}

\subsection{Consequences for formality.}\label{subsection:proof-formality} {The main goal of this subsection is to prove Theorem~\ref{thm:main-formality}. The proof is a formal consequence of Theorem~\ref{thm: dbpure is dbpure of mixed} combined with the decomposition result given in Lemma \ref{lem: dpure splitting underlying complex}.}

\begin{lm}\label{lem: dpure splitting underlying complex}
There is an equivalence of symmetric monoidal functors between the forgetful functor $\cD_{\pure}(\Rep_E^{\mix}WD_K)\to \cD(\Vect_{E})$ and its composition with the functor $M\mapsto \bigoplus\limits_{i\in\bZ} \hH^i(M)[-i]\colon   \cD(\Vect_E)\to \cD(\Vect_E)$.
\end{lm}

\begin{proof}
This lemma follows formally from the fact that the weight filtration on any mixed Weil--Deligne representation is naturally split, at the level of underlying vector spaces, as we now explain.

We denote by $\obl\colon \Rep_E^{\mix}WD_K\to \Vect_{E}$ the forgetful functor sending a Weil--Deligne representation to its underlying vector space. By \cite[Definition 2.13]{betts-litt}, this functor is isomorphic as a symmetric monoidal functor to the functor $V\mapsto \bigoplus\limits_i \gr^W_i V$. In particular, $\obl$ factors through the category $\Vect^{\bZ}_E$ of $\bZ$-graded vector spaces.

Passing to derived functors, we obtain a factorization of the symmetric monoidal functor $\cD(\obl):$
\[
\cD(\obl)\colon \cD(\Rep_E^{\mix}WD_K)\xrightarrow{\gr^{\bullet}_W} \cD(\Vect_E^{\bZ})\xrightarrow{\oplus} \cD(\Vect_E)
\]
Restricted to the full subcategory $\cD_{\pure}(\Rep^{\mix}_EWD_K)\subset \cD(\Rep_{E}^{\mix}WD_K)$, the functor $\gr^{\bullet}_W$ lands in the full subcategory $\cD_{\pure}(\Vect_E^{\bZ})$ consisting of graded complexes $M$ such that $\gr^i M$ is concentrated in cohomological degree $i$. This is a symmetric monoidal subcategory because tensor product over $E$ is exact.

However, the endofunctor $M\mapsto\bigoplus \hH^i(M)[-i]\colon \cD_{\pure}(\Vect^{\bZ}_E)\to \cD_{\pure}(\Vect^{\bZ}_E)$ is equivalent to the identity functor, as symmetric monoidal functors. Indeed, $\cD_{\pure}(\Vect^{\bZ}_E)$ is equivalent to the nerve of a symmetric monoidal $1$-category because all the mapping spaces in this category are discrete. Thus constructing an equivalence between two endofunctors amounts to doing so on the homotopy category of this $\infty$-category, where the equivalence is evident.

Since $\cD_{\pure}(\obl)\colon \cD_{\pure}(\Rep_E^{\mix}WD_K)\to\cD(\Vect_E)$ factors through $\cD_{\pure}(\Vect^{\bZ}_E)$, its composition with $M\mapsto\bigoplus \hH^i(M)[-i]$ is equivalent to $\cD_{\pure}(\obl)$ as a symmetric monoidal functor.
\end{proof}

Combining Lemma \ref{lem: dpure splitting underlying complex} with Theorem \ref{thm: dbpure is dbpure of mixed}, we get a general formality result:

\begin{cor}\label{cor: pure wd algebra is formal}
For any $E_{\infty}$-algebra $A\in \cD^b_{\pure}(\Rep_E WD_K)$ there exists an equivalence $A\simeq\bigoplus \hH^i(A)[-i]$ of $E_{\infty}$-algebras in $\cD(\Vect_E)$.
\end{cor}

It remains to apply Corollary \ref{cor: pure wd algebra is formal} to the cohomology of rigid-analytic spaces satisfying the weight-monodromy conjecture:

\begin{proof}[Proof of Theorem \ref{thm:main-formality}]

We have $\RGamma_{\et}(X_{\bC_p},\bQ_{\ell})$ as a commutative algebra object of $\cD(G_K;\bQ_{\ell})$. Each cohomology group $\hH^i_{\et}(X_{\bC_p},\bQ_{\ell})$ is a condensed representation of $G_K$ arising from a continuous representation on a finite-dimensional $\bQ_{\ell}$-vector space (cf. \cite[Proposition 4.9]{bosco}). Hence, Lemma \ref{lem: repgk derived of its heart} makes $\RGamma_{\et}(X_{\bC_p},\bQ_{\ell})$ into a commutative algebra object of the symmetric monoidal category $\cD^b(\Rep_{\bQ_{\ell}}G_K)$. Moreover, for $\ell=p$ the $p$-adic cohomology $\RGamma_{\et}(X_{\bC_p},\bQ_p)$ further lifts to an object of $\cD^b(\Rep^{\dR}_{\bQ_p}G_K)$ by Proposition \ref{prop: qp etale cohomology complex is de rham}.

The desired formality will now be a consequence of Corollary \ref{cor: pure wd algebra is formal}. Let us first treat the case $\ell\neq p$. Applying the symmetric monoidal functor $\DD\colon \cD^b(\Rep_{\bQ_{\ell}} G_K)\to \cD^b(\Rep_{\bQ_{\ell}}WD_K)$ from Example \ref{ex: wd from galois reps}(1), we obtain a commutative algebra object $\DD\bigl(\RGamma_{\et}(X_{\bC_p},\bQ_{\ell})\bigr)$ of $\cD^b(\Rep_{\bQ_{\ell}}WD_K)$ whose $i$-th cohomology is monodromy-pure of weight $i$, by the assumption of the theorem. That is, it lies in the full subcategory $\cD^b_{\pure}(\Rep_{\bQ_{\ell}}WD_K)$ and Corollary \ref{cor: pure wd algebra is formal} implies that $\DD(\RGamma_{\et}(X_{\bC_p},\bQ_{\ell}))$ is a formal algebra. Since the underlying commutative algebra object of $\cD(\Vect_{\bQ_{\ell}})$ of $\DD(\RGamma_{\et}(X_{\bC_p},\bQ_{\ell}))$ is equivalent to $\RGamma_{\et}(X_{\bC_p},\bQ_{\ell})$, we obtain the formality of $\ell$-adic cohomology for $\ell\neq p$.

Similarly, for $\ell=p$ we apply the functor from Example \ref{ex: wd from galois reps}(2) to the commutative algebra object $\RGamma_{\et}(X_{\bC_p},\bQ_p)\in \cD^b(\Rep_{\bQ_p}^{\dR}G_K)$ to obtain an object $\DD(\RGamma_{\et}(X_{\bC_p},\bQ_p))\in \cD_{\pure}^b(\Rep_{\bQ_p^{\nr}}WD_K)$. By de Rham comparison isomorphism, we have an equivalence of commutative algebras in $\cD(\Vect_{\BB_{\dR}})$
\begin{equation}
\RGamma_{\et}(X_{\bC_p},\bQ_p)\otimes_{\bQ_p} \BB_{\dR}\simeq\DD(\RGamma_{\et}(X_{\bC_p},\bQ_p))\otimes_{\bQ_p^{\nr}}\BB_{\dR}.
\end{equation}
Since $\DD(\RGamma_{\et}(X_{\bC_p},\bQ_p))$ is formal by Corollary \ref{cor: pure wd algebra is formal} and formality can be checked after an arbitrary field extension (see \cite[Corollary 6.9]{Obstructions}), we conclude that the $p$-adic \'etale cohomology algebra $\RGamma_{\et}(X_{\bC_p},\bQ_p)$ is formal as well.
\end{proof}

\bibliographystyle{alpha}
\bibliography{name}

\end{document}